\documentclass[a4paper,11pt,english,reqno]{amsart}
\usepackage{amsthm}
\usepackage{amsmath,amssymb,amsfonts,dsfont}    
\usepackage[english]{babel}
\usepackage{graphicx}
\usepackage{float}
\usepackage{stmaryrd}
\usepackage[a4paper, margin = 3cm, bottom = 3cm]{geometry}
\usepackage{enumerate}
\usepackage{marginnote}
\newcommand{\defeq}{\stackrel{\mathrm{def}}{=}}
\newcommand{\comp}{\mathrm{comp}}
\newcommand{\loc}{\mathrm{loc}}

\newcommand{\e}{\mathrm{e}}

\newcommand{\supp}{\mathrm{supp\,}}

\newcommand{\Ima}{\mathrm{Im\,}}
\newcommand{\Rea}{\mathrm{Re\,}}

\newcommand{\mO}{\mathcal{O}}
\newcommand{\C}{\mathds{C}}
\newcommand{\R}{\mathds{R}}
\newcommand{\N}{\mathds{N}}
\newcommand{\Z}{\mathds{Z}}

\newtheorem{thm}{Theorem}

\newtheorem{prop}[thm]{Proposition}
\newtheorem{lem}[thm]{Lemma}
\theoremstyle{definition}

\theoremstyle{remark}
\newtheorem{rem}[thm]{Remark}
\theoremstyle{definition}

\theoremstyle{definition}

\theoremstyle{definition}

\numberwithin{equation}{section} 
\usepackage{hyperref}		
\hypersetup{pdfborder=0 0 0, 
	    colorlinks=true,
	    citecolor=blue,
	    linkcolor=blue,
	    urlcolor=blue,
	    pdfauthor={Frederic Klopp, Martin Vogel}
	   } 
\title{Semiclassical resolvent estimates for bounded potentials}
\author{Fr{\'e}d{\'e}ric Klopp} \author{Martin Vogel} \address[Fr{\'e}d{\'e}ric
Klopp]{ \vskip.1cm Sorbonne Universit{\'e}, Universit{\'e} Paris Diderot,
  CNRS, Institut de Math{\'e}matiques Jussieu - Paris Rive Gauche ,
  F-75005, Paris, France}
\email{\href{mailto:frederic.klopp@imj-prg.fr}{frederic.klopp@imj-prg.fr}}
\address[Martin Vogel]{Mathematics Department, University of
  California, Berkeley, 887 Evans Hall, CA 94720, USA}
\email{\href{mailto:vogel@math.berkeley.edu}{vogel@math.berkeley.edu}}
\begin{document}
\begin{abstract}
  We study the cut-off resolvent of semiclassical
  Schr{\"o}dinger operators on $\R^d$ with bounded compactly supported
  potentials $V$. We prove that for real energies $\lambda^2$ in a
  compact interval in $\R_+$ and for any smooth cut-off function $\chi$
  supported in a ball near the support of the potential $V$, for some
  constant $C>0$, one has
  \begin{equation*}
    \| \chi (-h^2\Delta + V-\lambda^2)^{-1} \chi \|_{L^2\to H^1} \leq C
    \,\e^{Ch^{-4/3}\log \frac{1}{h} }. 
  \end{equation*}
  This bound shows in particular an upper bound on the imaginary parts
  of the resonances $\lambda$, defined as a pole of the meromorphic
  continuation of the resolvent $(-h^2\Delta + V-\lambda^2)^{-1}$ as
  an operator $L^2_{\comp}\to H^2_{\loc}$: any resonance $\lambda$ with
  real part in a compact interval away from $0$ has imaginary part at
  most
  \begin{equation*}
    \Ima \lambda \leq - C^{-1} \,\e^{Ch^{-4/3}\log \frac{1}{h}  }.
  \end{equation*}
  This is related to a conjecture by Landis: 
  The principal Carleman estimate in our proof provides as well a
  lower bound on the decay rate of $L^2$ solutions $u$ to
  $-\Delta u = Vu$ with $0\not\equiv V\in L^{\infty}(\R^d)$.  We show
  that there exist a constant $M>0$ such that for any such $u$, for
  $R>0$ sufficiently large, one has
  \begin{equation*}
    \int_{B(0,R+1)\backslash \overline{B(0,R)}}|u(x)|^2 dx 
    \geq M^{-1}R^{-4/3} \e^{-M \|V\|_{\infty}^{2/3} R^{4/3}}\|u\|^2_2.
  \end{equation*}
%
%
\end{abstract}
\maketitle
\setcounter{tocdepth}{1}
\tableofcontents
\section{Introduction}\label{sec:Int0}
In quantum mechanics the study of scattering systems naturally leads
to the study of \textit{quantum resonances} or \textit{scattering
  poles}, which can be defined as the complex-valued poles of the
meromorphic continuation of the scattering matrix or of the resolvent
of the Hamiltonian into the ``nonphysical sheet" of the complex plane.
They can also be seen as a generalization of eigenvalues of a bounded
system in which energy can scatter to infinity. A typical associated
resonance state has then a rate of oscillation and a rate of decay or 
``inverse life-time" which can be associated to the imaginary part of
the resonance.  In wave scattering for instance, one can describe the
long-time dynamics of a wave, scattered on an obstacle or a potential,
via the resonances and the associated resonant states.  It is then the
resonances closest to real axis, i.e. those with the longest
``life-time", whose contribution in the scattered wave ``survives" the
longest. Therefore, the study of the resonances close to the real axis
is in some sense the most pertinent one.
%
%
\\
\par
We consider the semiclassical Schr{\"o}dinger operator on
$\R^d$
\begin{equation}\label{int1}
  P_V \defeq -h^2\Delta + V 
\end{equation}
where $h\in (0,1]$ denotes the semiclassical parameter and the potential
$V\in L^{\infty}_{\comp}(\R^d;\R)$ is real-valued with compact support.
The potential $V(x)=V(x;h)$ may depend on $h>0$. However, in this case
we suppose that
\begin{equation}\label{int1.0}
  \| V\|_{\infty} =C_V < + \infty
\end{equation}
and that the support of $V$ is contained in the ball
$B(0,R_0)\Subset\R^d$ of radius $R_0>0$, with both constants $C_V>0$
and $R_0>0$ independent of $h>0$.
\subsection{Resolvent bounds}
We prove the following resolvent estimate:
\begin{thm}\label{thm:1}
  Let $I$ be a compact interval in $\R\backslash \{0\}$. Let $R>R_0$
  and assume that the dimension $d\geq 2$. Then there exists constants
  $C>0$ and $h_0\in (0,1]$, so that for all $0<h\leq h_0$, all
  $v\in L^2_{comp}(B(0,R))$ and all $\lambda \in I$
  \begin{equation}\label{thm:1.1}
    \| (P_{V}-\lambda^2)^{-1} v \|_{H^1(B(0,R))} \leq  C\, \e^{Ch^{-4/3}\log \frac{1}{h}}\|v\|_2.
  \end{equation}
\end{thm}
In dimension $d=1$ a stronger result is known: there we have that
\begin{equation}\label{thm:1.1.0}
  \| (P_{V}-\lambda^2)^{-1} v \|_{H^1(B(0,R))} \leq C\, \e^{Ch^{-1}}\|v\|_2, 
\end{equation}
see for instance the proof in \cite[Theorem 2.29]{DZ17+}. From our
proof of Theorem \ref{thm:1} in dimension $d\geq2$ we get actually
that the statement holds when we replace $H^1(B(0,R))$ on the left
hand side of \eqref{thm:1.1} by $H^1(B(0,Rh^{-1/3})$ for any $R>0$.
\par Equivalently, we can formulate the statement of Theorem
\ref{thm:1} as an estimate on the cut-off resolvent. More precisely,
we have for any $\chi \in \mathcal{C}^{\infty}_c(\R^d)$ with $0\leq \chi\leq 1$
and $\chi\equiv 1$ near $B(0,R_0)$ that there exist constants $C>0$
and $h_0 \in (0,1]$ such that
\begin{equation}\label{thm:1.1b}
  \| \chi (P_{V}-\lambda^2)^{-1} \chi \|_{L^2\to H^1} \leq C
  \,\e^{Ch^{-4/3}\log \frac{1}{h}}. 
\end{equation}
\par
Shapiro \cite{Sh18+} obtained independently from our
work a quantitative limiting absorption principle for $P_V$, with
$V \in L^{\infty}_{\comp}(\R^d;\R)$, in dimension $d\geq 1$. Shapiro
proved that for fixed positive energy $E>0$ and $s>1/2$ one has for
$h>0$ small enough and any $\varepsilon >0$ that
\begin{equation}\label{thm:1.1bb}
  \| \langle |x|\rangle^{-s} (P_{V}-E- i\varepsilon )^{-1}\langle |x|\rangle^{-s}  \|_{L^2\to H^2} \leq \e^{Ch^{-4/3}\log h^{-1}},
\end{equation}
for some constant $C>0$ depending only on the $L^{\infty}$ norm of
$V$, the energy $E$, the dimension $d$ and $s$.
\par
For any other resolvent estimates so far, one assumed at least that
not only $V$ but also the radial derivative $\partial_r V$ are
bounded: Datchev \cite{Da14} proved a quantitative limiting absorption
principal in dimension $d\neq 2$ for $L^{\infty}$ potentials $V$ with
radial derivative $\partial_r V \in L^{\infty}$ satisfying the decay
conditions $V\leq\langle r\rangle^{-\delta_0}$ and
$\partial_r V \leq \langle r\rangle^{-1-\delta_0}$, i.e.
\begin{equation}
  \label{thm:1.1d}
  \| \langle |x|\rangle^{-s} (P_{V}-E- i\varepsilon )^{-1}\langle |x|\rangle^{-s}  \|_{L^2\to L^2} \leq \e^{C_1h^{-1}},
\end{equation}
for $E>0$, any $s>1/2$, $h>0$ small enough and any $\varepsilon >0$. In dimension $d=2$ 
Shapiro \cite{Sh16} proved \eqref{thm:1.1d} replacing the above assumptions on $\partial_r V$ with $\nabla V \in L^{\infty}$ 
and $|\nabla V |\leq \langle r\rangle^{-1-\delta_0}$. Vodev
\cite{Vod14} proved a bound similar to \eqref{thm:1.1d} for potentials
satisfying the decay conditions
$\sup_{\R^d} \langle x\rangle^{1+\delta} |V(x,h)| \leq C h^{\nu}$ and
$\partial_r V \leq Ch^{\nu}\langle r\rangle^{-1-\delta}$ for some
constants $C,\nu,\delta>0$. Dyatlov and Zworski \cite{DZ17+}
simplified Datchev's proof for $V, \partial_r V \in L^{\infty}_{\comp}$
in dimension $d\neq 2$ and showed
\begin{equation}\label{thm:1.1c}
  \| \chi (P_{V}-\lambda^2)^{-1} \chi \|_{L^2\to L^2} \leq C_0 \,\e^{C_1 h^{-1}}.
\end{equation}
Similar results were proven for various cases of short-range and
long-range perturbations of the Laplacian $-h^2\Delta$ under stronger
regularity assumptions. Burq \cite{Bu98,Bu02} proved \eqref{thm:1.1c}
for smooth $V$ decaying sufficiently quickly near infinity and on
domains of type $\R^d\backslash \overline{O}$ for some compact
obstacle $O$ with smooth boundary.  Different proofs of Burq's
theorem, providing some simplifications and extensions were given by
Vodev \cite{Vo00} and Sj{\"o}strand \cite{Sj02ln}. Moreover, 
Cardoso and Vodev \cite{CaVo02}  provide a version of Burq's theorem on a class 
of infinite volume Riemannian manifolds with cusps. 
\subsection{Resonance free regions}
As a consequence of Theorem \ref{thm:1} we get that there is a
resonance free region below the real axis away from $0$.
\par
There are various ways of defining \textit{resonances} of a quantum
Hamiltonian, see for instance \cite{DZ17+} for an overview. One way is
to define them as the poles of the a meromorphic continuation of the
resolvent through the essential spectrum. More precisely, we have the
following well-known result \cite[Proposition 2.1]{Sj02ln} and \cite[Theorem 3.6]{DZ17+}:
\begin{thm}\label{thm:Mc2a}
  The meromorphic family of operators
  \begin{equation*}
    (P_V - \lambda^2)^{-1} : ~ L^2_{\comp}(\R^d) \longrightarrow
    H^2_{\loc}(\R^d),   \quad \Ima \lambda >0,
  \end{equation*}
  has a meromorphic extension from the upper half-plane
  $\Ima \lambda >0$ to
  \begin{enumerate}
  \item $\lambda \in\C\backslash\{0\}$, when $d=1$,
  \item $\lambda \in\C $, when $d\geq 3$ is odd,
  \item $\lambda$ in the logarithmic covering space of
    $\C\backslash\{0\}$, when $d\geq 2$ is even.
  \end{enumerate}
\end{thm}
The \textit{resonances} of $P_V$ are then defined as the poles of this
extension with possibly the exception of the $L^2$ eigenvalues of
$P_V$ situated on the imaginary axis $i[0,+\infty)$. See Section
\ref{sec:Res} below for more details.
\par We prove the following
\begin{thm}\label{thm:2}
  Let $I$ be a compact interval in $\R\backslash \{0\}$ and suppose
  that $d\geq 2$, then there exists constants $C>0$ and $h_0\in (0,1]$
  such that for $ 0< h \leq h_0$ there are no resonances of $P_V$ in the
  set of $\lambda \in \C$ with
  \begin{equation*}
    \Rea \lambda \in I, \quad \Ima \lambda \geq - C^{-1}\,\e^{-Ch^{-4/3}\log
      \frac{1}{h}}. 
  \end{equation*}	
\end{thm}
In the case of dimension one $d=1$ we have a stronger result: there
exist constants $C_0,C_1>0$ and $h_0\in (0,1]$ such that for $ 0< h \leq h_0$
there are no resonances of $P_V$ in the set of $\lambda \in \C$ with
\begin{equation*}
  \Rea \lambda \in I, \quad \Ima \lambda \geq - C^{-1}\,\e^{-Ch^{-1}},
\end{equation*}	
see for instance \cite[Theorem 2.29]{DZ17+}. This bound is optimal as
can be seen for the study of resonances for cut off random potentials
\cite{Kl16}.
\subsection{Remark on Landis' conjecture and decay of eigenfunctions}
We do not think that the bounds in \eqref{thm:1.1} and in Theorem
\ref{thm:2} are optimal.  The $h^{-4/3}$ in the exponent comes from a
Carleman estimate in a ball $B(0,R)$ with $R>R_0$ which cannot
distinguish between real-valued and complex-valued potentials, see
Lemma \ref{lem:3.1} below. Yet in the proof of Theorem \ref{thm:1} we
crucially use that the potential $V$ is assumed to be real-valued in
flux norm estimate on outgoing solution in Lemma \ref{lem:Out}. We now
present a slightly modified version of our main Carleman estimate:
\begin{lem}(see Lemma \ref{lem:3.1})\label{lem:1.1} Let $P_V$ be as in
  \eqref{int1} with $V\in L^{\infty}_{\comp}(\R^d,\C)$ a bounded
  (possibly) complex valued potential with compact support satisfying
  \eqref{int1.0}. Let $I\Subset \R$ be a compact interval.  Let
  $R>R_0$. Then, there exists a real-valued smooth function
  $\phi \in\mathcal{C}^{\infty}(\R^d)$ and a constants $C>0$ and
  $h_0\in (0,1]$, such that for all $u\in\mathcal{C}_c^{\infty}(B(0,R))$,
  all $\lambda\in I$ and all $0<h\leq h_0$
  \begin{equation*}
    \int \e^{2\phi/h^{4/3}}(|u|^2 + |h\nabla u|^2 ) dx 
    \leq \frac{C}{h^{2/3}}
    \int \e^{2\phi/h^{4/3}}| (P_V - \lambda^2)u|^2 dx. 
  \end{equation*}	
\end{lem}
Here, the exponent $h^{-4/3}$ is optimal, since there we can allow for
complex-valued potentials. This can be seen from a counter example to
Landis conjecture \cite{KoLa88} by Meshkov \cite{Me92}: Landis
conjectured that if $u$ is a bounded solution to $-\Delta u + Vu=0$ in
$\R^d$, with $\|V\|_{\infty} = \mO(1)$ and
$|u(x)| \leq C\exp(-c|x|^{1+})$, then $u\equiv 0$.  The conjecture holds
in dimension $d=1$ which is consistent with the estimate
\eqref{thm:1.1.0}.  However, in general this conjecture was disproved
by Meshkov \cite{Me92} who constructed a complex-valued bounded
potential $V$ and a complex-valued function $u$ which solve
$-\Delta u + Vu=0$ in $\R^2$ such that $|u(x)|\leq C\exp(-c|x|^{4/3})$,
$x\in\R^2$. Meshkov \cite{Me92} also proved a quantitative unique
continuation principal: if $u$ is a bounded solution to
$-\Delta u + Vu=0$ and decays faster than $\exp(-\tau |x|^{4/3})$ for
any $\tau>0$ as $|x|\to +\infty$, then necessarily $u\equiv 0$.
\par As a consequence of Lemma \ref{lem:3.1} we get the following
lower bound on the decay of eigenfunctions of Schr{\"o}dinger operator
with $L^{\infty}$ potentials.
\begin{thm}\label{thm3}
  Let $\|V\|_{\infty} \leq C_V$ with $V\not\equiv 0$. Then, there exist a
  constant $M>0$ such that for any solution $u$ to
  \begin{equation}\label{RL0}
    -\Delta u + V u = 0 \quad \text{in } \R^d
  \end{equation}
  satisfying $\|u\|_2 =1$, for $R>0$ sufficiently large
  \begin{equation}\label{RL1}
    \int_{B(0,R,R+1)}|u(x)|^2 dx \geq M^{-1}R^{-4/3} \e^{-M
      \|V\|_{\infty}^{2/3} R^{4/3}}, 
  \end{equation}
  where $B(0,R,R+1)=B(0,R+1)\backslash \overline{B(0,R)}\Subset \R^d$
  denotes the annulus of inner radius $R$ and outer radius $R+1$
  centered at $0$.
\end{thm}
If in Lemma \ref{lem:1.1} we had a weight $\exp(2\phi h^{-4/3+})$,
then this would imply a corresponding lower bound
$\exp(-M \|V\|_{\infty}^{2/3} R^{4/3-})$ in \eqref{RL1} which would be
in contradiction with Meshkov's counter example to Landis' conjecture.
\par Let us remark that Bourgain and Kenig \cite{BoKe05} proved the
following more local estimate for $u$, a solution to \eqref{RL0},
\begin{equation}\label{RL2}
  \int_{B(j,1)}|u(x)|^2 dx \geq C \e^{-c|j|^{4/3}\log|j|}, \quad \text{for } |j|\to +\infty.
\end{equation}
The lower bound \eqref{RL1} is a slight improvement over \eqref{RL2}
since we loose the logarithm yet we pay the price of taking averages
in a large annulus rather than in a small ball.
\par In a series of works by Naki\'c, T{\'a}ufer, Tautenhahn and Veseli\'c
\cite{NTTV15,NTTV18} a scale free unique continuation principal was
proven. The authors consider an equidistributed sequence of balls
$B(z_j,\delta)$ centered at $z_j\in \R^d$, with $j\in \Z^d$, and of radius
$\delta \in (0,G/2)$, for some $G>0$, so that
$B(z_j,\delta)\Subset (-G/2,G/2)^d +j$. They showed that there exists
a constant $N=N(d)>0$ depending only on the dimension $d$, such that
for all $G>0$, all $\delta \in (0,G/2)$, all equidistributed sequences of
balls as above, all $V\in L^{\infty}(\R^d;\R)$, all $L\in G\N$, any energy
$E_0\geq 0$ and all
$\phi \in \mathrm{ran}(\mathbf{1}_{(-\infty,E_0]}(H|_{\Lambda_L}))$
\begin{equation}\label{RL3}
  \|\phi\|^2_{S_{\delta}\cap \Lambda_L} 
  \geq \left(\frac{\delta}{G}\right)^{N(1+G^{4/3}\|V\|_{\infty}^{2/3}+G\sqrt{E_0})}
  \|\phi\|^2_{ \Lambda_L},
\end{equation}
where $S_{\delta} = \bigcup_{j\in (G\Z)^d}B(z_j,\delta)$ and
$\Lambda_L = (-L/2,L/2)^d$.  This results extends previous results by
Rojas-Molina and Veseli\'c \cite{RMV13}, Combes, Hislop and Klopp
\cite{CHK07} and Klein \cite{Kl13}.
\par
Tautenhahn and Veseli\'c \cite{TV15} extended the above result to
$\psi\in \mathrm{ran}(\mathbf{1}_{I}(H))$, for any interval
$I\subset (-\infty,E_0]$, i.e.
\begin{equation}\label{RL4}
  \|\psi\|^2_{S_{\delta}} 
  \geq \frac{1}{2}\left(\frac{\delta}{G}\right)^{N(1+G^{4/3}(2\|V\|_{\infty}+E_0)^{2/3})}
  \|\psi\|^2_{ \R^d},
\end{equation}
In a recent paper by Borsiv, Tautenhahn and Veseli\'c \cite{BTV17} a
more general scale free unique continuation principal was proven for
second order elliptic differential operators.
\par
It is striking that in the above results the dependence of the
exponent on the potential is only $ \|V\|_{\infty}^{2/3}$. This agrees
very well with our results \eqref{RL1}. However, we do not know
whether this dependence is optimal.
\\
\par
Meshkov's example uses fundamentally that the potential is
complex-valued.  Since Lemma \ref{lem:3.1} below cannot distinguish
between real-valued and complex-valued potentials, we cannot improve
the exponent $h^{-4/3}$ in Theorem \ref{thm:1} with our method in
spite of the fact that, there, the potential is assumed to be
real-valued which is crucial for a flux norm estimate on outgoing
solution in Lemma \ref{lem:Out} below.
\\
\par
Finally, let us remark that Landis' conjecture may still hold true for
real-valued bounded potentials $V$ and real-valued functions $u$. In
fact some recent developments have been made by Davey, Kenig and Wang
\cite{DKW17} in dimension $d=2$.
%
%
%
%
%
%
\subsection{Notation}
Let $\chi_1,\chi_2\in\mathcal{C}_c^{\infty}(\R^d;[0,1])$. When we
write $\chi_1\prec\chi_2$, we mean that $\chi_2\equiv 1 $ in a small
neighborhood of the support of $\chi_1$. We extend this definition in
the obvious way to include indicator functions of open sets.
\par Depending on the context we will denote by $|x |$ norm of $x$ as
a vector in some Banach space or the absolute value of $x$ as a
complex variable. Similarly, we will denote by $(x |y)$ the inner
product of $x,y$ as elements of some vector space.
\subsection*{Acknowledgments}
The second author was supported by the Erwin Schr{\"o}dinger Fellowship
J4039-N35, by the National Science Foundation grant DMS-1500852 and 
by CNRS Momentum. We would like to thank Maciej Zworski and Nicolas 
Burq for a very helpful and encouraging discussion. We would also like to 
thank the anonymous referee for his pertinent remarks to help improve this 
paper. 
\section{Meromorphic continuation of the resolvent}\label{sec:Res}
Let $h\in (0,1]$ be the semiclassical parameter and consider the operator
\begin{equation}\label{eq3.0.3}
  P_V-\lambda^2= -h^2\Delta + V(x;h) -\lambda^2 \quad \text{on }L^2(\R^d), 
\end{equation}
where $V = V(\cdot; h) \in L^{\infty}_{\comp}(\R^d,\R)$ is a bounded
real-valued compactly supported potential which may depend on the
semiclassical parameter $h>0$. We will often suppress the dependence on
$h$ and simply write $V$. We assume
\begin{equation}\label{eq3.2}
  \| V\|_{\infty}  \leq C_V < +\infty
\end{equation}
and that the support of $V$ is contained in the ball
$B(0,R_0)\Subset \R^d$ of radius $R_0>0$,
\begin{equation}\label{eq3.2.1}
  \supp V \subset B(0,R_0)\Subset \R^d, 
\end{equation}
where both constants $C_V>0$ and $R_0>0$ are are independent of $h>0$.
Moreover, we assume that $\lambda$ is in a compact interval $I$ away
from $0$, i.e. we suppose that
\begin{equation}\label{eq3.2.2}
  \lambda \in I=[a,b] \Subset\R \backslash\{0\}.
\end{equation}
Since the potential $V$ is bounded and has compact support, it follows
that the essential spectrum of $P_V$ is given by $[0,+\infty)$ and
that in $(-\infty,0)$ are only isolated eigenvalues of finite
multiplicity.
\par
For let $\lambda \in \C$ with $\Ima \lambda >0$ the resolvent
\begin{equation}\label{mc2}
  R(\lambda)\defeq 
  (P_{V}- \lambda^2)^{-1} : ~ L^2(\R^d) \longrightarrow L^2(\R^d) 
\end{equation}
is a bounded linear operator. In this notation, we find the negative
eigenvalues of $P_{V}$ on $i\R_+$ given by $\lambda_j=i\mu_j$.
\subsection{Holomorphic continuation of the resolvent of the free
  Laplacian $P_0$}
Seen as an operator $L^2_{\comp}(\R^d) \to H^2_{\loc}(\R^d)$, it is
possible to meromorphically continue the resolvent across the real
axis. In the following we will recall some well-known results. We
begin with the meromorphic continuation of the free resolvent
\begin{equation}\label{mc3}
  R_0(\lambda)\defeq 
  (-h^2\Delta- \lambda^2)^{-1} : ~ L^2(\R^d) \longrightarrow L^2(\R^d), 
  \quad \Ima \lambda >0. 
\end{equation}
\begin{thm}\label{thm:Mc1}
  The family of operators
  \begin{equation*}
    R_0(\lambda)=
    (-h^2\Delta- \lambda^2)^{-1} : ~ L^2_{\comp}(\R^d) \longrightarrow H^2_{\loc}(\R^d), 
    \quad \Ima \lambda >0,
  \end{equation*}
  has a holomorphic extension from the upper half-plane
  $\Ima \lambda >0$ to
  \begin{enumerate}
  \item $\lambda \in\C\backslash\{0\}$, when $d=1$,
  \item $\lambda \in\C $, when $d\geq 3$ is odd,
  \item $\lambda$ in the logarithmic (universal) covering space of
    $\C\backslash\{0\}$, when $d\geq 2$ is even.
  \end{enumerate}
  Moreover, for any $\Omega\Subset\C\backslash\{0\}$ and any
  $\chi\in\mathcal{C}_c^{\infty}(\R^d)$ there exist constants
  $C_0,C_1>0$ such that for all $\lambda\in \Omega$ and $h>0$ small
  enough
  \begin{equation}\label{mc3.1}
    \| \chi R_0(\lambda)\chi \|_{L^2\to H^1} \leq C_0\,\e^{C_1/h}.
  \end{equation}
\end{thm}
\begin{proof}
  See for instance \cite[Section 2.1]{Sj02ln}, \cite[Theorem 3.1]{DZ17+}.
\end{proof}
In dimension $d=1$ the free resolvent $R_0(\lambda)$ has a simple pole
at $\lambda=0$.  It can be extended meromorphically to the
entire plane $\C$. However, in this paper we will be interested in
energies away from $0$, therefore we will not need this particular
result.
\subsection{Meromorphic continuation of the resolvent of $P_V$}
When adding a bounded potential $V$ with compact support we can no
longer extend the resolvent $R(\lambda)$ holomorphically since poles
appear. More precisely, we have the following result.
\begin{thm}\label{thm:Mc2}
  The family of operators
  \begin{equation*}
    R_V(\lambda)\defeq 
    (-h^2\Delta + V- \lambda^2)^{-1} : ~ L^2_{\comp}(\R^d) \longrightarrow H^2_{\loc}(\R^d), 
    \quad \Ima \lambda >0,
  \end{equation*}
  has a meromorphic extension from the upper half-plane
  $\Ima \lambda >0$ to
  \begin{enumerate}
  \item $\lambda \in\C\backslash\{0\}$, when $d=1$,
  \item $\lambda \in\C $, when $d\geq 3$ is odd,
  \item $\lambda$ in the logarithmic (universal) covering space of
    $\C\backslash\{0\}$, when $d\geq 2$ is even.
  \end{enumerate}
\end{thm}
\begin{proof}
  See for instance \cite[Proposition 2.1]{Sj02ln}, \cite[Theorem 3.6]{DZ17+}.
\end{proof}
By definition, resonances or scattering poles of $P_V$ are the poles
of this extension with exception of the $L^2$ eigenvalues of $P_V$ at
$\lambda = i\mu_j$.
\\
\par
Let $\Omega$ be an open set in $\C$ or in a covering surface over some
open set in $\C$. Then we say that a function $\Omega \ni z \mapsto P(z)$
with values in the space of linear operators $L^2_{\comp}\to H^2_{\loc}$
is \textit{holomorphic} if $\chi_1P(z)\chi_2$ is holomorphic as a
function with values in the space of bounded linear operators
$L^2\to H^2$, for all $\chi_j\in\mathcal{C}_c^{\infty}$.
\par
Correspondingly, we say that a function $\Omega \ni z \mapsto P(z)$ with
values in the space of linear operators $L^2_{\comp}\to H^2_{\loc}$ is
\textit{meromorphic} if it is holomorphic on $\Omega\backslash S$,
where $S$ is a discrete subset of $\Omega$, and such that if $z_0\in S$,
then near $z_0$ we have
\begin{equation*}
  P(z) = \sum_{j=1}^N\frac{A_j}{(z-z_0)^j} + B(z)
\end{equation*}
where $N$ is finite, $B(z)$ is a holomorphic function with values in
the space of linear operators $L^2_{\comp}\to H^2_{\loc}$ and
$A_j:~L^2_{\comp}\to H^2_{\loc}$ are of finite rank and continuous, in
the sense that $\chi_1A_j \chi_2$ is bounded for all
$\chi_j\in\mathcal{C}_c^{\infty}$.
\subsection{Absence of resonances on the real axis}
We end this section by recalling the following result.
\begin{prop}
  $P_V$ has no resonances in $\R\backslash\{0\}$.
\end{prop}
\begin{proof}
  The proof is based on the fact that $\lambda_0$ is a resonance of
  $P_V$ if and only if there exists a solution $u$ to
  $(P_V-\lambda_0^2)u =0$ of the form $u=R_0(\lambda_0)w$ for some
  $w\in L^2_{\comp}$. Such solutions are called outgoing solutions.  This
  is then combined with the Paley-Wiener theorem and the Carleman
  estimate in Lemma \ref{lem:3.1} below to conclude the result. One can 
  follow line by line (using Lemma \ref{lem:3.1}) the standard proof which 
  can be found for instance in \cite[Theorem 2.4]{Sj02ln}, 
  \cite[Theorem 3.30]{DZ17+}.
\end{proof}
\section{Resolvent estimate}\label{sec:NoRes}
In this section we will present the proof of Theorem \ref{thm:1}. The
global strategy of this proof was inspired by the approach to Carleman
estimates in \cite[Section 4]{Sj02ln}.
\subsection{Local Carleman estimate in a ball}
From now on we suppose that $d\geq 2$ and we work under the assumption
\eqref{eq3.2} and \eqref{eq3.2.1}. The first step in the proof of
Theorem \ref{thm:1} is to give a local Carleman estimate in a ball.
\begin{lem}\label{lem:3.1}
  Let $I\Subset \R$ be a compact interval. Then, for any $R>0$, there
  exists a real-valued smooth function
  $\phi \in\mathcal{C}^{\infty}(\R^d)$ and constants $C>0$ and
  $h_0\in (0,1]$, such that for all $u\in\mathcal{C}_c^{\infty}(B(0,R))$,
  all $\lambda\in I$ and all $0<h\leq h_0$
  \begin{equation*}
    \int \e^{2\phi/h^{4/3}}(|u|^2 + |h\nabla u|^2 ) dx 
    \leq \frac{C}{h^{2/3}}
    \int \e^{2\phi/h^{4/3}}| (P_V - \lambda^2)u|^2 dx. 
  \end{equation*}	
\end{lem}
\begin{proof}
  The basic Carleman estimate \cite[Lemma 4.2]{Sj02ln} for the semiclassical
  Laplacian $-\widetilde{h}^2\Delta$ is as follows: Let $R>0$. Then,
  there exists a smooth real-valued function
  $\phi\in\mathcal{C}^{\infty}(\R^d)$, and constants $C_0 >0$ and
  $\widetilde{h}_0\in\,(0,1]$ such that for all
  $v\in\mathcal{C}_c^{\infty}(B(0,R))$ and all
  $0<\widetilde{h}\leq \widetilde{h}_0$
  \begin{equation}\label{eq3.n1}
    \widetilde{h} \int (|v|^2 + |\widetilde{h}\nabla v|^2) dx 
    \leq C_0 \| \e^{\phi/\widetilde{h}} (-\widetilde{h}^2\Delta ) \e^{-\phi/\widetilde{h}}u\|^2.
  \end{equation}	
  Next, let $h\in (0,1]$ and let $C_1>0$ be so that
  $C_1 \geq 4 C_0\max\{ \|V-\lambda^2\|^2_{\infty},1\}$ for all
  $\lambda \in I$. Define
  \begin{equation}\label{eq3.n2}
    P_V(\widetilde{h}) \defeq  - \widetilde{h}^2\Delta +\left(\frac{\widetilde{h}}{C_1}\right)^{1/2}(V- \lambda^2)
  \end{equation}
  with $\widetilde{h}=h^{4/3}C_1^{-1/3}$. Notice that
  \begin{equation}\label{eq3.n3}
    P_V(\widetilde{h}) = \left(\frac{h}{C_1}\right)^{2/3}(P_V - \lambda^2).
  \end{equation}
  Then, by \eqref{eq3.n1}, \eqref{eq3.n2}, we have that for all
  $0<\widetilde{h}\leq \min\{\widetilde{h}_0,C_1^{-1/3}\}$, for all
  $v\in\mathcal{C}_c^{\infty}(B(0,R))$ and any $\lambda \in I$
  \begin{equation}\label{eq3.n4}
    \begin{split}
      \| \e^{\phi/\widetilde{h}} P_V(\widetilde{h})
      \e^{-\phi/\widetilde{h}}v\|_2 &\geq \| \e^{\phi/\widetilde{h}}
      (-\widetilde{h}^2\Delta ) \e^{-\phi/\widetilde{h}}v\|_2
      -\frac{\widetilde{h}^{1/2}}{C_1^{1/2}} \|(V- \lambda^2)v\|_2 \\
      & \geq \frac{\widetilde{h}^{1/2}}{C_0^{1/2}}\left(\int (|v|^2 +
        |\widetilde{h}\nabla v|^2 ) dx\right)^{1/2}
      -\frac{\widetilde{h}^{1/2}}{C_1^{1/2}} \|(V- \lambda^2)\|_{\infty} \|v\|_2 \\
      & \geq\frac{\widetilde{h}^{1/2}}{2C_0^{1/2}}\left(\int (|v|^2 +
        |\widetilde{h}\nabla v|^2 ) dx\right)^{1/2}.
    \end{split}
  \end{equation}	
%
  Setting $u = \e^{\phi/\widetilde{h}} v$, we get by \eqref{eq3.n4}
  \begin{equation}\label{eq3.n5}
    \int \e^{2\phi/\widetilde{h}}(|u|^2 + |\widetilde{h}\nabla u|^2 ) dx 
    \leq \frac{C}{\widetilde{h}}
    \int \e^{2\phi/\widetilde{h}}| P_V(\widetilde{h}) u|^2 dx.
  \end{equation}	
  for some constant $C>0$. Set $\widetilde{\phi} = C_1^{1/3}\phi
  $. Then, by \eqref{eq3.n3}, \eqref{eq3.n5}
  \begin{equation}\label{eq3.n6}
    \int \e^{2\widetilde{\phi}/h^{4/3}}(|u|^2 + |h\nabla u|^2 ) dx 
    \leq \frac{C}{h^{2/3}}
    \int \e^{2\widetilde{\phi}/h^{4/3}}| (P_V - \lambda^2)u|^2 dx, 
  \end{equation}	
  which concludes the proof of the Lemma.
\end{proof}
Next we will get rid of the assumption of compact support on $u$ in
Lemma \ref{lem:3.1}.  Suppose that $R_0<R_1<R_2$, let
$u\in\mathcal{C}^{\infty}(B(0,R_2))$, let
$1_{B(0,R_1)} \prec\chi\in\mathcal{C}^{\infty}_c(B(0,R_2),[0,1])$ and
apply Lemma \ref{lem:3.1} to $\chi u$ to get
\begin{equation}\label{eq3.16}
  \begin{split}
    \int_{B(0,R_1)}\e^{2\phi/h^{4/3}}(|u|^2 + |h\nabla u|^2) dx
    \leq& \frac{C}{h^{2/3}} \int_{B(0,R_2)}\e^{2\phi/h^{4/3}}|(P_V-\lambda)u|^2dx \\
    &+ \frac{C}{h^{2/3}}
    \int_{B(0,R_2)}\e^{2\phi/h^{4/3}}|[-h^2\Delta,\chi] u|^2 dx.
  \end{split}
\end{equation}
We denote by $B(0,R_1,R_2)\subset \R^d$ the open annulus
$B(0,R_2)\backslash \overline{B(0,R_1)}$. Since
\begin{equation*}
  \begin{split}
    |[-h^2\Delta,\chi] u|^2 &= |(-h^2\Delta\chi)u - 2(h\nabla\chi | h\nabla u) |^2 \\
    &\leq C_1 (h^4|u|^2 + h^2 | h\nabla u|^2),
  \end{split}
\end{equation*}
for some constant $C_1>0$, and since $\supp \nabla\chi \subset B(0,R_1,R_2)$,
we obtain from \eqref{eq3.16}
\begin{equation}\label{eq3.17}
  \begin{split}
    \int_{B(0,R_1)}\e^{2\phi/h^{4/3}}(|u|^2 + |h\nabla u|^2) dx
    &\leq \frac{C}{h^{2/3}} \int_{B(0,R_2)}\e^{2\phi/h^{4/3}}|(P_V-\lambda)u|^2dx \\
    &+ CC_1h^{4/3} \int_{B(0,R_1,R_2)}\e^{2\phi/h^{4/3}}(|u|^2 +
    |h\nabla u|^2) dx.
  \end{split}
\end{equation}
\subsection{Carleman estimate in a shell away from the support of the
  potential $V$}
We will begin with the following
\begin{lem}\label{lem:3.2a}
  Let $w=w(r)=r^2$ for $r\geq 0$. Let $I$ be as in \eqref{eq3.2.2} and
  let $\lambda\in I$. Let $A,B>0$ be constants (to be determined later
  on) and set
  \begin{equation}\label{eq3.23.2}
    R_c = R_c(h) = \frac{\sqrt{2A}}{h^{1/3} |\lambda |}. 
  \end{equation}
  Let $h>0$ be small enough so that $2 R_0 \leq R_c$. Then, there exists a
  smooth real-valued function
  $\phi_0\in\mathcal{C}^{\infty}(]0,+\infty[)$ and a constant
  $C_{\phi_0}>0$ (independent of $h>0$) so that
  $0\leq \phi'_0|_{[R_0,+\infty[} \leq C_{\phi_0}$ and
  \begin{equation}\label{eq3.23.3}
    \phi'_0(r) = \begin{cases}
      \left(A r^{-2}-\frac{h^{2/3}\lambda^2}{2}\right)^{1/2}, \quad \text{for} ~ 
      R_0\leq r \leq R_c-2, \\
      {B}^{-1}h^{1/3}, \quad \text{for} ~ r \geq R_c-1.
    \end{cases}
  \end{equation}
  Moreover, there exists a constant $h_0\in(0,1]$ and $C_0>0$,
  depending only on $A,B,I$ and $R_0$, so that for any $0<h\leq h_0$
  \begin{equation}\label{eq3.23.0}
    (w(h^{2/3}\lambda^{2}+ (\phi_0')^2-h^{4/3}\phi_0''))' \geq h^{2/3}\frac{\lambda^2 w'}{C_0}, 
    \quad\text{ for } ~r \geq R_0.
  \end{equation}
\end{lem}
\begin{proof}
  \textbf{Step 1}. Set $\psi = (\phi'_0)^2$. To simplify the notation
  we will suppose that $\lambda > 0$ and we will work with
  $\widetilde{h}=h^{2/3}$ with $0<h \leq h_0$ for some $h_0\in (0,1]$.  Then,
  \eqref{eq3.23.0} is equivalent to
  \begin{equation}\label{eq3.23}
    G(r)\defeq 
    \psi +\widetilde{h}\lambda^2 -\widetilde{h}^2\phi_0'' 
    +  \frac{r}{2}(\psi' -\widetilde{h}^2\phi_0''') \geq 
    \widetilde{h}\frac{\lambda^2}{C_0}, 
    \quad 	\text{ for } r\geq R_0 .
  \end{equation}
  Let $A>0$, let $h_0\in (0,1]$ be small enough so that
  \begin{equation}\label{Z1}
    2R_0 \leq R_c = \frac{\sqrt{2A}}{ \widetilde{h}^{1/2} \lambda},  
  \end{equation}
  compare with \eqref{eq3.23.2}. Set
  \begin{equation}\label{eq3.23.1}
    \psi_0(r) \defeq \frac{A}{r^{2}}-\frac{\widetilde{h}\lambda^2}{2}, \quad r> 0.
  \end{equation}
  Notice that $\psi_0(r)=0$ precisely at $r=R_c$. Hence, for
  $h_0\in (0,1]$ small enough, $\psi_0\geq 0$ for $0<r\leq R_c$.
  \par
  Next, let $B>0$ and set
  \begin{equation}\label{eq3.24.0}
    \psi_1 \defeq\frac{\widetilde{h}}{B^2}.
  \end{equation}
  Let $\chi\in\mathcal{C}^{\infty}([0,+\infty[;[0,1])$ be so that
  $\chi\equiv 1$ on $[0,R_c-2]$, $\chi\equiv 0$ on $[R_c-1,+\infty[$
  and so that all derivatives of $\chi$ are bounded uniformly in $h$
  (and, thus, $\widetilde{h}$).  We can choose $\chi$ such that
  $\chi' \leq 0$. Then, set
  \begin{equation}\label{eq3.24.1}
    \psi(r) \defeq \psi_0(r)\chi(r) + \psi_1(r)(1-\chi(r)), \quad r>0
  \end{equation}
  Since $\psi_0(R_c)=0$, we have that for $\tau\in[-2,-1]$
  \begin{equation}\label{eq3.24.2}
    \begin{split}
      \psi_0(R_c+\tau) &= -2 A\tau \int_0^1 (R_c+t\tau)^{-3}dt
      =  -2 A\tau R_c^{-3}(1 + \mO(R_c^{-1}))\\
      &= \frac{|\tau| \widetilde{h}^{3/2}\lambda^3}{(2A)^{1/2}} (1 +
      \mO(\widetilde{h}^{1/2}))
    \end{split}
  \end{equation}
  Since $\psi_0$ is a strictly decreasing function, by
  \eqref{eq3.24.1}, \eqref{eq3.24.0}, there exist constants
  $h_0\in (0,1]$ and $C>0$ (depending on A,B and I) such that for all
  $0<h\leq h_0$
  \begin{equation}\label{eq3.24.3}
    \psi(r) \geq \frac{1}{C}\, \widetilde{h}^{3/2}, \quad r>0 
  \end{equation}
  \textbf{Step 2}. We estimate $\phi''_0$. Assume first that
  $R_0 \leq r\leq R_c-2$. Then,
  \begin{equation}\label{eq3.24}
    \phi''_0(r) = \frac{\psi_0'(r)}{2\sqrt{\psi_0(r)}} 
    = \frac{-\sqrt{2}A}{\sqrt{2Ar^{4} - \widetilde{h}\lambda^2r^{6}}}
    \defeq \frac{-\sqrt{2}A}{m(r)^{1/2}} <0.
  \end{equation}
  Notice that $m'(r)= r^3(8A - \widetilde{h} \lambda^2 6r^2)$. Thus,
  $m(r)$ has its unique critical point at
  \begin{equation*}
    r_1= \frac{2\sqrt{A}}{\lambda\sqrt{3\widetilde{h}}} = \sqrt{\frac{2}{3}} R_c < R_c
  \end{equation*}
  where by \eqref{Z1} we have that $r_1 > \sqrt{2}R_0$.  Hence, for
  $h_0\in (0,1]$ small enough, $m'(r)>0$ on $[R_0,r_1[$ and $m'(r)<0$ on
  $]r_1,+\infty[$.  This implies that $m(r)^{-1/2}$ is decreasing on
  $[R_0,r_1]$ and increasing on $[r_1,R_c-2]$. Therefore, $|\phi_0''|$
  is bounded by the maximum of $|\phi_0''(R_0)|$, $|\phi_0''(r_1)|$
  and $|\phi_0''(R_c-2)|$.
  \\
  \par
  By \eqref{eq3.24}, for $h_0\in (0,1]$ small enough there exists a
  constant $C>0$ (depending as well on $I$, $A$ and $R_0$) such that
  \begin{equation*}
    |\phi_0''(R_0)|\leq C
  \end{equation*}
  A straight forward computation shows that
  \begin{equation*}
    \phi''_0(r_1) =  
    -\frac{3\sqrt{3A}}{2R_c^2}	= -\frac{3\sqrt{3}\,\widetilde{h} \lambda^2}{4\sqrt{A}}=
    \mO_{A,I}(\widetilde{h})
  \end{equation*}
  and Taylor expansion shows that
  \begin{equation*}
    \phi''_0(R_c-2) =  \mO_{A,I}(\widetilde{h}^{1/4})
  \end{equation*}
  for all $0<h\leq h_0$ with $h_0\in (0,1]$ small enough.
  \begin{rem}
    If the constant in the big O notation depends on one of the
    parameters mentioned in the hypotheses of Lemma \ref{lem:3.2a},
    then we add them as subscripts to keep track of the dependencies.
  \end{rem}
  In conclusion, we have that for all $0<h\leq h_0$, with $h_0\in (0,1]$
  small enough,
  \begin{equation}\label{eq3.25}
    \phi_0''(r) =  \mO_{A,I,R_0}(1), \text{  for  } R_0\leq r\leq R_c-2.
  \end{equation}
  Next, suppose that $ r \in [ R_c-2,R_c-1]$. There,
  \begin{equation}\label{eq3.26}
    \phi''_0 = \frac{\psi_0'\chi + (\psi_0-\psi_1)\chi'}{2\sqrt{\psi}}.
  \end{equation}
  By \eqref{eq3.23.1},
  \begin{equation*}
    |\psi'_0(r)| \leq \frac{2 A}{ R_c^{3}}(1 + \mO(R_c^{-1}))
    \leq \mO_{A,I}(\widetilde{h}^{3/2}).
  \end{equation*}
  Since $\chi'\leq 0$, by \eqref{eq3.24.2} for all $h >0$, sufficiently
  small, we have
  \begin{equation}\label{eq3.26.1}
    0 \leq (\psi_0-\psi_1)\chi' \leq  \mO_{A,B,I}(\widetilde{h}).
  \end{equation}
  Combining the above two estimates with \eqref{eq3.26} and
  \eqref{eq3.24.3}, we get that
  \begin{equation}\label{eq3.26.1.0}
    \phi''_0(r) = \mO(\widetilde{h}^{1/4}), \quad  r \in [ R_c-2,R_c-1].
  \end{equation}
  Notice that $\phi_0''(r)=0$ for $r\geq R_c-1$. Then, putting this
  together \eqref{eq3.25} and \eqref{eq3.26.1.0}, we that for all
  $0<h\leq h_0$, with $h_0\in (0,1]$ small enough,
  \begin{equation}\label{eq3.27}
    \phi''_0(r) =\begin{cases}
      \mO_{A,I,R_0}(1), \quad R_0\leq r\leq R_c-2, \\
      \mO_{A,B,I}(\widetilde{h}^{1/4}), 
      \quad  R_c-2\leq r \leq R_c-1, \\
      0, \quad  r\geq R_c-1.
    \end{cases}
  \end{equation}
  \textbf{Step 3}. Recall that $\psi=(\phi'_0)^2$. Hence, by
  \eqref{eq3.24.3}
  \begin{equation}\label{Z2}
    -r \phi'''_0 = -r \frac{\psi''}{2\sqrt{\psi}} + \frac{r(\psi')^2}{4\psi^{3/2}} 
    \geq -r \frac{\psi''}{2\sqrt{\psi}} \defeq -f.
  \end{equation}
  We will show that $f$ is bounded. Suppose first that
  $r\in[R_0,R_c-2]$. There, by \eqref{eq3.24.1}
  \begin{equation*}
    f(r) = \frac{r\psi_0''}{2\sqrt{\psi_0}} = \frac{3\sqrt{2}A}{m(r)^{1/2}} \geq 0,
  \end{equation*}
  with $m$ as in \eqref{eq3.24}. Considering the critical point of
  $f$, as in the discussion following \eqref{eq3.24}, we get that
  $f(r)$ is bounded by the maximum of $f(R_0)$, $f(r_1)$ and
  $f(R_c-2)$. Performing similar computations as for \eqref{eq3.25},
  we get that for all $0<h\leq h_0$, with $h_0\in (0,1]$ small enough,
  \begin{equation*}
    0\leq f(r) \leq \mO_{A,I,R_0}(1) \quad \text{for } r\in[R_0,R_c-2].
  \end{equation*}
  Next, suppose that $r\in[R_c-2,R_c-1]$. By \eqref{eq3.24.1},
  \begin{equation*}
    r\psi'' = r(\psi_0'' \chi + 2\chi'\psi'_0 + (\psi_0-\psi_1)\chi'').
  \end{equation*}
  We will estimate each term separately. First, using \eqref{eq3.23.1}
  and Taylor expansion, we see that
  \begin{equation*}
    |r\psi_0''| = \frac{6A}{R_c^{3}}(1+\mO(R_c^{-1})) 
    = \mO_{A,I}(\widetilde{h}^{3/2})
  \end{equation*}
  and
  \begin{equation*}
    |r\psi_0'| = \frac{6A}{R_c^{2}}(1+\mO(R_c^{-1})) 
    = \mO_{A,I}(\widetilde{h}).
  \end{equation*}
  By \eqref{eq3.24.2}, we get that for all $0<h\leq h_0$, with $h_0\in (0,1]$
  small enough,
  \begin{equation*}
    r|(\psi_0-\psi_1)\chi''| \leq r\,\mO_{I,B}(\widetilde{h}) 
    \leq \mO_{A,B,I}(\widetilde{h}^{1/2}).
  \end{equation*}
  Combining the above three estimates with \eqref{eq3.24.3} and
  \eqref{Z2}, we have that for all $0<h\leq h_0$, with $h_0\in (0,1]$ small
  enough,
  \begin{equation*}
    |f(r)| \leq \mO_{A,B,I}(\widetilde{h}^{-1/4}), 
    \quad \text{for } r\in[R_c-2,R_c-1].
  \end{equation*}
  Finally notice that $\phi'''_0(r)=0$ for $r\geq R_c-1$. Therefore,
  \begin{equation}\label{eq3.28}
    -\widetilde{h}r \phi'''_0(r) \geq \begin{cases}
      \mO_{A,I,R_0}(\widetilde{h}), \quad R_0\leq r\leq R_c-2, \\
      \mO_{A,I,B}
      (\widetilde{h}^{3/4}), \quad  R_c-2\leq r \leq R_c-1, \\
      0, \quad  r\geq R_c-1.
    \end{cases}
  \end{equation}
  \textbf{Step 4}. We check that $\psi$, see \eqref{eq3.24.1}, with
  $\psi=(\phi_0')^2$ satisfies \eqref{eq3.23}.  Suppose first that
  $r\in[R_0,R_c-2]$. By \eqref{eq3.23.1},
  \begin{equation*}
    \psi_0 + \frac{r}{2}\psi_0' = -\frac{\widetilde{h}\lambda^2}{2}.
  \end{equation*}
  Then, by \eqref{eq3.23} \eqref{eq3.28}, \eqref{eq3.27} and
  \eqref{eq3.23.1} there exist constants $h_0\in (0,1]$ and $C_1>0$
  (depending on $A$, $I$, $R_0$ and $B$) such that for all $0<h\leq h_0$
  \begin{equation}
    \begin{split}
      G(r) &= \psi_0(r) +\widetilde{h}\lambda^2
      -\widetilde{h}^2\phi_0'' (r)
      +  \frac{r}{2}(\psi_0'(r)-\widetilde{h}^2\phi_0'''(r))\\
      & \geq \frac{\widetilde{h}\lambda^2}{2}
      + \mO_{A,I,R_0}(\widetilde{h}^2)  \\
      & \geq \frac{\lambda^2 \widetilde{h}}{C_1}.
    \end{split}
  \end{equation}
  \par
  Next, assume that $r\in[R_c-2,R_c-1]$. Then, by \eqref{eq3.23},
  \eqref{eq3.28}, \eqref{eq3.27}, \eqref{eq3.24.3}, \eqref{eq3.26.1}
  and \eqref{eq3.24.1} there exist constants $h_0\in (0,1]$ and $C_2>0$
  (depending on $A$, $I$, $R_0$ and $B$) such that for all $0<h\leq h_0$
  \begin{equation}
    \begin{split}
      G &\geq \psi_0\chi +\psi_1(1-\chi) + \widetilde{h}\lambda^2
      +\frac{r}{2}\psi_0'\chi +\frac{r}{2}(\psi_0-\psi_1)\chi' \\
      &\phantom{\geq}+ \mO_{A,B,I}(\widetilde{h}^{2+1/4})+
      \mO_{A,B,I}(\widetilde{h}^{3/2 + 1/4}) \\
      &\geq \frac{1}{C}\, \widetilde{h}^{3/2}
      +\frac{\widetilde{h}\lambda^2}{2}+
      \mO_{A,B,I}(\widetilde{h}^{1 + 3/4}) \\
      &\geq \frac{\lambda^2\widetilde{h}}{C_2}.
    \end{split}
  \end{equation}
  \par
  Finally, suppose that $r\geq R_c-1$. Then, by \eqref{eq3.28},
  \eqref{eq3.27} and \eqref{eq3.24.1},
  \begin{equation}
    G(r) = \frac{\widetilde{h}}{B^2}+\widetilde{h}\lambda^2 
  \end{equation}
  In conclusion, $\psi$ is a positive smooth function on $]0,+\infty[$
  and satisfies \eqref{eq3.23}.
\end{proof}
\begin{lem}\label{lem:3.2}
  Let $I$ be as in \eqref{eq3.2.2}. Let $R_3>R_2$ and let
  $P_0=-h^2\Delta$.  Let $\phi_0$, $C_{\phi_0}>0$, $A>0$ and $B>0$ be
  as in Lemma \ref{lem:3.2a}. Then, there exists a constant
  $C=C(I,R_0,A,B,C_{\phi_0})>0 $ and an $h_0 \in\, (0,1]$ such that
  for all $u\in \mathcal{C}^{\infty}_c(B(0,R_0,R_3))$ and all $0<h\leq h_0$
  \begin{equation}\label{eq3.29}
    \int \e^{2\phi_0/h^{4/3}}(|u|^2 + |h\nabla u|^2) dx 
    \leq\frac{C R_3^3}{h^{2+2/3} }
    \int \e^{2\phi_0/h^{4/3}}|(P_0-\lambda^2)u|^2dx,
  \end{equation}
  where we write $\phi_0=\phi_0(|x|)$.
\end{lem}
\begin{proof}
  The proof is an adaption of the proof of a global Carleman estimate
  by Datchev \cite{Da14}. We begin by passing to spherical
  coordinates, where
  \begin{equation*}
    -h^2\Delta = -h^2\partial^2_r - \frac{d-1}{2r}h^2\partial_r - r^{-2}h^2\Delta_{S^{d-1}},
  \end{equation*}
  where $-\Delta_{S^{d-1}} \geq 0$ denotes the Laplace-Beltrami
  operator on the $(d-1)$-dimensional sphere $S^{d-1}$. Set
  \begin{equation}\label{eq3.6.1}
    P_{\phi_0} \defeq 
    \e^{\phi_0/h^{4/3}}
    r^{\frac{d-1}{2}}
    (P_0-\lambda^2)
    r^{-\frac{d-1}{2}}
    \e^{-\phi_0/h^{4/3}}.
  \end{equation}
  A straight forward computation shows that
  \begin{equation}\label{eq3.6}
    P_{\phi_0} =  -h^2\partial^2_r + 2\phi_0'h^{2/3}\partial_r +V_{\phi_0} + \Lambda - \lambda^2,
  \end{equation}
  where $\phi_0' = \partial_r\phi_0$ and
  \begin{equation}\label{eq3.6a}
    V_{\phi_0} \defeq h^{2/3}\phi_0'' - h^{-2/3}(\phi_0')^2 
  \end{equation}
  and
  \begin{equation}\label{eq3.6.2}
    \Lambda = h^2r^{-2}\left(-\Delta_{S^{d-1}}+ \frac{d-1}{2}\frac{d-3}{2}\right)
  \end{equation}
  which is a positive semidefinite operator for $d\geq 3$, and
  $\geq -\frac{h^2}{4r^2}$ for $d =2$.  Next, set $w=w(r)= r^2$ and
  let $f' = \partial_r f$ denote the radial derivative, and write for
  $v\in\mathcal{C}_c^{\infty}(B(0,R_0,R_3))$
  \begin{equation}\label{eq3.7}
    F(r) \defeq \|hv'(r\cdot)\|_{S^{d-1}}^2 
    - ( (\Lambda + V_{\phi_0} - \lambda^2)v(r\cdot)|v(r\cdot))_{S^{d-1}}, \quad r>0,
  \end{equation}
  where the norm and the scalar product are the norm and scalar
  product of $L^2(S^{d-1})$. Since the support of $v$ is compact, we
  have that
  \begin{equation}\label{eq3.8}
    \int_0^{\infty} (w(r)F(r))' dr =0.
  \end{equation}
  Since $\Lambda$ is self-adjoint, we get by \eqref{eq3.6},
  \begin{equation}\label{eq3.8.1}
    \begin{split}
      F' &= 2\Rea (h^2\partial_r^2v|v')_{S^{d-1}}
      - 2\Rea( (\Lambda + V_{\phi_0} -\lambda^2)v|v')_{S^{d-1}} \\
      &\phantom{=}+2r^{-1}(\Lambda v|v)_{S^{d-1}} - ( V_{\phi_0}'v|v)_{S^{d-1}} \\
      &= -2\Rea (P_{\phi_0} v|v')_{S^{d-1}} + 4h^{-4/3}\phi_0' \|hv'\|_{S^{d-1}}^2 \\
      &\phantom{=}+2r^{-1}(\Lambda v|v)_{S^{d-1}} - (
      V_{\phi_0}'v|v)_{S^{d-1}}.
    \end{split}
  \end{equation}
  Recall that we are working in $0<R_0\leq r\leq R_3$ and that
  $w=r^2$. Therefore, $w\phi'_0\geq 0$ and $2r^{-1}w-w'= 0$. Then,
  using as well the elementary inequality
  $\|a\|^2-2\Rea(a|b) + \|b\|^2\geq 0$, we get that
  \begin{equation}\label{eq3.9}
    \begin{split}
      (wF)'
      &= -2w\Rea ( P_{\phi_0} v|v')_{S^{d-1}} + (4h^{-4/3}w\phi_0'+w')\|hv'\|_{S^{d-1}}^2 \\
      &\phantom{=}+(2wr^{-1}-w')(\Lambda v|v)_{S^{d-1}}
      +( (w(\lambda^2-V_{\phi_0}))'v|v)_{S^{d-1}} \\
      &\geq -\frac{r^3}{2h^2}\|P_{\phi_0} v\|^2_{S^{d-1}} +(
      (w(\lambda^2-V_{\phi_0}))'v|v)_{S^{d-1}}.
    \end{split}
  \end{equation}
  Integrating \eqref{eq3.9} with respect to $r$, we get by
  \eqref{eq3.8}, \eqref{eq3.23.0} and \eqref{eq3.6a}
  \begin{equation}\label{eq3.10}
    \int_0^{\infty}\int_{S^{d-1}} |v|^2 dr d\sigma 
    \leq 
    \frac{C_0 R_3^3}{4R_0 \lambda^2h^2 }\int_0^{\infty}
    \int_{S^{d-1}}|P_{\phi_0} v|^2dr d\sigma. 
  \end{equation}
  Here, we used as well that $\supp v \subset B(0,R_0,R_3)$. Moreover, recall
  from Lemma \ref{lem:3.2a} that the constant $C_0$ depends only on
  the energy interval $I$ and the constants $A,B,R_0$.
  \par
  Setting $u = \e^{\phi_0/h^{4/3}}r^{(d-1)/2}v$, we get by
  \eqref{eq3.6.1} that
  \begin{equation}\label{eq3.14}
    \int \e^{2\phi_0/h^{4/3}} |u|^2 dx
    \leq
    \frac{C_0R_3^3}{4R_0 \lambda^2h^2 } \int \e^{2\phi_0/h^{4/3}}|(P_0-\lambda^2) u|^2dx 
  \end{equation}
  Integration by parts yields that
  \begin{equation}\label{eq3.14.1}
    \int \e^{2\phi_0/h^{4/3}} |h\nabla u|^2 dx = 
    -\Rea \int h \,\mathrm{div}(\e^{2\phi_0/h^{4/3}}h\nabla u)\overline{u} dx.
  \end{equation}
  The right hand side is bounded from above by
  \begin{equation}\label{eq3.14.2}
    \begin{split}
      -\int \e^{2\phi_0/h^{4/3}}&
      2\Rea\left(\sqrt{2}\phi'_0(|x|)\overline{u}h^{-1/3} \frac{x}{|x|} \bigg| \frac{h}{\sqrt{2}}\nabla u\right) dx \\
      &+ \int \e^{\phi_0/h^{4/3}} |(P_0-\lambda^2)u|\,|u| dx
      +\lambda^2\int \e^{\phi_0/h^{4/3}}|u|^2 dx.
    \end{split}
  \end{equation}
  Using the elementary inequality $2\Rea(a|b)\leq |a|^2+|b|^2$, we get by
  \eqref{eq3.14.1}, \eqref{eq3.14.2} that
  \begin{equation}\label{eq3.33a}
    \begin{split}
      \int \e^{2\phi_0/h^{4/3}} |h\nabla u|^2 dx &\leq 
      (2\lambda^2+1+4\|\phi'\|_{\infty}^2h^{-2/3} )\int \e^{2\phi_0/h^{4/3}} | u|^2 dx \\
      &+\int \e^{2\phi_0/h^{4/3}}|(P_0-\lambda^2) u|^2dx.
    \end{split}
  \end{equation}
  Let $\lambda_{\infty}$ denote the minimum of the absolute value of
  the supremum and infimum of the interval $I$. Then, by
  \eqref{eq3.33a}, \eqref{eq3.14}, we have that, for $h>0$ small
  enough,
  \begin{equation}\label{eq3.15.1}
    \int \e^{2\phi_0/h^{4/3}}(|u|^2 + |h\nabla u|^2) dx 
    \leq\frac{4C_0 R_3^3 \|\phi'_0\|_{\infty}^2}{R_0 \lambda_{\infty}^2 h^{2+2/3} }
    \int \e^{2\phi_0/h^{4/3}}|(P_0-\lambda^2)u|^2dx.
  \end{equation}
  Recall from Lemma \ref{lem:3.2a} that
  $\|\phi'_0|_{[R_0,+\infty[}\|_{\infty} \leq C_{\phi_0}$. This concludes
  the proof of Lemma \ref{lem:3.2}.
\end{proof}
\subsection{Combining Carleman estimates}
Next, let $R_3=R_3(h) \asymp h^{-1/3}$ and let
$u\in\mathcal{C}^{\infty}(B(0,R_3))$ so that
\begin{equation}\label{eq3.34}
  (P_V-\lambda^2)u = v\in\mathcal{C}_c^{\infty}(B(0,R))
\end{equation}
and suppose that $R_0 < R < R_1-2$. Recall \eqref{eq3.17} and set
$M=\phi(R_2)$.  Then,
\begin{equation}\label{eq3.35}
  \begin{split}
    \int_{B(0,R_1)}(|u|^2 + |h\nabla u|^2) dx
    &\leq \frac{C \e^{2M/h^{4/3}}}{h^{2/3}} \int |v|^2dx \\
    &+ Ch^{4/3}\e^{2M/h^{4/3}} \int_{B(0,R_1,R_2)}(|u|^2 + |h\nabla
    u|^2) dx.
  \end{split}
\end{equation}
Let
$1_{B(0,R_1-1,R_3-1)}\prec\chi\in\mathcal{C}^{\infty}_c(B(0,R_1-2,R_3);[0,1])$
so that all derivatives of $\chi$ are bounded (uniformly in
$h$). Applying \eqref{eq3.29} to $\chi u$, we obtain similar to
\eqref{eq3.17} that
\begin{equation}\label{eq3.36}
  \begin{split}
    \int_{B(0,R_1,R_3-1)} \e^{2\phi_0/h^{4/3}}(|u|^2 + |h\nabla u|^2)
    dx &\leq 
    \frac{C R_3^3}{h^{2/3}} \int_{B(0,R_1-2,R_1-1)}\e^{2\phi_0/h^{4/3}}(|u|^2 + |h\nabla u|^2)  dx\\
    & + \frac{C R_3^3}{h^{2/3}}
    \int_{B(0,R_3-1,R_3)}\e^{2\phi_0/h^{4/3}}(|u|^2 + |h\nabla u|^2)
    dx.
  \end{split}
\end{equation}
Here we used as well that 
\begin{equation*}
	(P_0-\lambda^2)u=(P_V-\lambda^2)u=v=0, \quad \text{on } B(0,R_1-2,R_3),
\end{equation*}
which follows from \eqref{eq3.34} and the assumption that $\supp V\Subset B(0,R_0)$, 
see the discussion after \eqref{int1.0}.
\par
Recall \eqref{eq3.23.3} and let $\eta>0$. Then, by shifting $\phi_0$
by a constant and by choosing $A>0$ large enough, we can arrange that,
for $h>0$ small enough,
\begin{equation*}
  \begin{split}
    &\phi_0(|x|)\leq -\eta, \quad \text{for } |x| \leq R_1-1, \\
    &\phi_0(|x|)\geq M, \quad \text{for } |x| \geq R_1.
  \end{split}
\end{equation*}
Thus,
\begin{itemize}
\item the second term on the right hand side of \eqref{eq3.35} is
  bounded by the a constant times the left hand side of
  \eqref{eq3.36};
\item the first term on the right hand side of \eqref{eq3.36} is
  bounded by a factor $\mO(\e^{-1/Ch})$ times the left hand side of
  \eqref{eq3.35}.
\end{itemize}
Therefore, adding \eqref{eq3.35} and \eqref{eq3.36} we get for $h>0$
small enough
\begin{equation}\label{eq3.37}
  \begin{split}
    \int_{B(0,R_3-1)}\e^{2\psi/h^{4/3}}(|u|^2 + |h\nabla u|^2) dx
    &\leq e^{2M/h^{4/3}} \frac{C}{h^{2/3}} \int |v|^2dx \\
    & + \frac{C R_3^3}{h^{2/3}}
    \int_{B(0,R_3-1,R_3)}\e^{2\psi/h^{4/3}}(|u|^2 + |h\nabla u|^2) dx,
  \end{split}
\end{equation}
with
\begin{equation}\label{eq3.38}
  \psi(x)=\begin{cases}
    0, \quad \text{for } |x| \leq R_1, \\
    \phi_0(x), \quad \text{for } |x| \geq R_1.
  \end{cases}
\end{equation}
\subsection{Outgoing solutions and flux norm}
Now assume that
\begin{equation}\label{eq3.50}
  u=R_V(\lambda)v
\end{equation}
with $v\in L^2_{\comp}(B(0,R))$ is an outgoing solution with $R>0$ as
above.  By Theorem \ref{thm:Mc2} and analytic continuation we see that
$u$ satisfies $(P_V-\lambda^2)u = v$. Moreover, by a density argument,
we see that $u$ and $v$ satisfy \eqref{eq3.37}. In particular, since
$u$ is outgoing, there exists a $w\in L^2_{\comp}(B(0,R))$ so that
\begin{equation*}
  u(x) = R_0(\lambda)w(x), \quad \text{for } |x|\geq R_1. 
\end{equation*}
Hence, $u$ is a solution to the free Helmholtz equation
$(-h^2\Delta -\lambda^2)u=0$ outside the ball $B(0,R_1)$.
\\
\par
Let $R_c\defeq \widetilde{R_c}(\lambda)h^{-1/3}$ be as in Lemma
\ref{lem:3.2a}. Recall \eqref{eq3.2.2}, let $C_r>1$ be a constant and
set
\begin{equation}\label{eq3.38.0}
  R_3 \defeq \widetilde{R}_3h^{-1/3} \defeq C_r \widetilde{R_c}(a)h^{-1/3}.
\end{equation}
Recall \eqref{eq3.23.3} and write for $r\geq R_c$
\begin{equation}\label{eq3.38a}
  \begin{split}
    \phi_0(r) &= A^{1/2} \int_{R_0}^{R_c-2}(t^{-2} - R_c^{-2})^{1/2} dt +
    \int_{R_c-2}^{R_c-1} \phi_0'(t) dt
    + \int_{R_c-1}^{r}B^{-1}h^{1/3}dt \\
    &\defeq I_1 + I_2 + I_3.
  \end{split}
\end{equation}
The first integral in \eqref{eq3.38a} is bounded by
\begin{equation*}
  |I_1| \leq \frac{A^{1/2}}{3} \log \frac{1}{h} + A^{1/2}\log\frac{\sqrt{2 A} - 2 |\lambda| h^{1/3} }{|\lambda|\,R_0}.
\end{equation*}
The second integral $|I_2| \leq C_{\phi_0}$, see Lemma \ref{lem:3.2a}, and
the third integral $I_3 = h^{1/3}B^{-1}(r-R_c+1)$. Hence,
%
\begin{equation}\label{eq3.38.1}
  \phi_0(|x|) = C_0(h) + \frac{h^{1/3}}{B}|x|,\quad  \text{for } |x|\geq R_c
\end{equation}
where $C_0(h)$ depends on $A,I,B,R_0,C_{\phi_0}$ and $h>0$ satisfying
\begin{equation}\label{eq3.38.2}
  |C_0(h)| \leq \frac{A^{1/2}}{3} \log \frac{1}{h} + \mO_{A,I,B,R_0}(1)
\end{equation}
for $h>0$ small enough. Using Lemma \ref{lem:Out} below and
\eqref{eq3.38.0}, we see that for $C_r>1$ large enough, the second
term on the right hand side of \eqref{eq3.37} is bounded from above by
\begin{equation}\label{eq3.38.3}
  \begin{split}
    \mO(h^{-5/3})\e^{2\psi(R_3)/h^{4/3}}&\int_{B(0,R_3-1,R_3)}(|u|^2 + |h\nabla u|^2)  dx \\
    & \leq  \mO(h^{-3})\e^{2\psi(R_3)/h^{4/3}} \Ima (v | u) \\
    &\phantom{\leq}+ \mO(h^{-5/3}) \e^{(2\psi(R_3)-\delta )/ h^{4/3}}
    \int_{A(\widetilde{R}_3/4,1,h)}(|u|^2 + |h\nabla u|^2)\,dx.
  \end{split}
\end{equation}
for some $\delta>0$. Using \eqref{eq3.38.1} and \eqref{eq3.38.0} we
get that
\begin{equation*}
  \begin{split}
    2\psi(R_3) -\delta = 2C_0+\frac{1}{B}(2\widetilde{R}_3 - B\delta)
    \leq 2\psi(|x|)-\frac{\delta_1}{B}, \quad \text{for } |x|\geq
    (\widetilde{R}_3/4-1)h^{-1/3}.
  \end{split}
\end{equation*}
where in the second to the last inequality we chose $B>0$ to be large
enough so that
$2\widetilde{R}_3-\delta B \leq 2(\widetilde{R}_3/4 -1) - \delta_1$
for some $\delta_1>0$.  Hence, for $h>0$ small enough, we can absorb
the second term on the right hand side of \eqref{eq3.37} into the term
on the left hand side of \eqref{eq3.37}. Hence,
\begin{equation}\label{eq3.38.4}
  \begin{split}
    \int_{B(0,R_3-1)}\e^{2\psi/h^{4/3}}(|u|^2 + |h\nabla u|^2) dx
    &\leq e^{2M/h^{4/3}} \frac{C}{h^{2/3}} \int |v|^2dx \\
    & +\mO(h^{-3})\e^{2\psi(R_3)/h^{4/3}} \Ima (v | u) .
  \end{split}
\end{equation}
By the Cauchy-Schwartz inequality and \eqref{eq3.50} we get that
\begin{equation}\label{eq3.38.5}
  \begin{split}
    \Ima (v|u) &\leq \|v\|\, \|u\|_{L^2(B(0,R))} \\
    &\leq \frac{h^3}{C} \e^{-2\psi(R_3)/h^{4/3}} \| u \|_{L^2(B(0,R))}^2 +
    Ch^{-3} \e^{2\psi(R_3)/h^{4/3}} \| v \|^2 .
  \end{split}
\end{equation}
In view of \eqref{eq3.38},\eqref{eq3.38.1}, \eqref{eq3.38.2} by
\eqref{eq3.38.4} and \eqref{eq3.38.5} there exists constant $C,C'>0$
such that
\begin{equation*}
  \int_{B(0,R_3-1)}(|u|^2 + |h\nabla u|^2) dx 
  \leq C' e^{Ch^{-4/3} \log\frac{1}{h}} \int |v|^2dx
\end{equation*}
which together with \eqref{eq3.50} concludes the proof of Theorem
\ref{thm:1}.
\begin{lem}\label{lem:Out}
  Assume \eqref{eq3.50}. Then, for any $\widetilde{R} >0$ (independent
  of $h>0$) and any $0< \eta< 3\widetilde{R}/16$ there exist constants
  $C,C',\delta,h_0>0$ such that for any $\lambda \in I$ and any
  $0<h <h_0$
  \begin{equation*}
    \begin{split}
      \int_{A(\widetilde{R},\eta,h)}(|u|^2 + |h\nabla u|^2)\,d x
      &\leq C h^{-4/3}\Ima (v | u) \\
      &+ C' \e^{-\delta /
        h^{4/3}}\int_{A(\widetilde{R}/4,\eta,h)}(|u|^2 + |h\nabla
      u|^2)\,dx,
    \end{split}
  \end{equation*}
  where
  $A(\widetilde{R},\eta,h) \defeq
  B(0,(\widetilde{R}-\eta)h^{-1/3},(\widetilde{R}+\eta)h^{-1/3})$.
\end{lem}
\begin{proof}
  Let $U_{h}: L^2(\R^d) \to L^2(\R^d)$ be the unitary map defined by
  \begin{equation}\label{eq3.39}
    (U_{h}\phi)(x) = h^{-d/6}\phi(h^{-1/3}x). 
  \end{equation}
%
  Using \eqref{eq3.0.3} we rescale the operator $P_V-\lambda^2$ by
  $h^{-1/3}$, i.e.
  \begin{equation}\label{eq3.40}
    \begin{split}
      U_{h}(P_V - \lambda^2) U^*_{h}&=
      -h^{2+2/3}\Delta + V(h^{-1/3}x;h)  - \lambda^2 \\
      &\defeq -\widetilde{h}^2\Delta + \widetilde{V}(x;h) - \lambda^2 \\
      &\defeq (\widetilde{P}_{\widetilde{V}} - \lambda^2).
    \end{split}
  \end{equation}
  Let $u$ be as in \eqref{eq3.50}. As discussed there, $u$ is a
  solution to the free Helmholtz equation
  $(-h^2\Delta-\lambda^2)u= 0$ outside the ball
  $B(0,R_1)$. Set $\widetilde{u}\defeq U_{h}u$.  Then, we have
  that outside the ball $B(0,R_1h^{1/3})$
  \begin{equation}\label{eq3.41}
    (\widetilde{P}_0-\lambda^2)\widetilde{u} = 0. 
  \end{equation}
  Hence, by \cite[Proposition 2.2]{Bu98}, it follows that for any
  $\widetilde{R}_2>\widetilde{R}_1>0$ (constants independent of $h>0$)
  there exist $C,C',\delta,h_0>0$ such that for any $\lambda\in I$ and
  any $0<h\leq h_0$
  \begin{equation}\label{eq3.42}
    \begin{split}
      -\Ima \int_{r=\widetilde{R}_2} \widetilde{h}\partial_r
      \widetilde{u}\cdot\overline{\widetilde{u}}\,d\sigma \geq &~
      C\lambda \int_{r=\widetilde{R}_2}(|\widetilde{u}|^2
      + \lambda^{-2} |\widetilde{h}\nabla \widetilde{u}|^2)\,d\sigma \\
      &
      -C'\e^{-\delta|\lambda|/\widetilde{h}}\int_{r=\widetilde{R}_1}(|\widetilde{u}|^2
      + \lambda^{-2} |\widetilde{h}\nabla \widetilde{u}|^2)\,d\sigma,
    \end{split}
  \end{equation}
  where $d\sigma$ is the surface measure on
  $\partial B(0,\widetilde{R}_2)$, respectively on
  $\partial B(0,\widetilde{R}_1)$, induced from the Lebesgue measure
  on $\R^d$. Let $\eta>0$ be as in the hypothesis, then
  $\widetilde{R}/4+\eta < \widetilde{R}/2-\eta$. The mean value
  theorem implies that there exists a
  $\widetilde{R}_2\in [\widetilde{R}-\eta,\widetilde{R}+\eta]$ such that
  \begin{equation}\label{eq3.43}
    \int_{r=\widetilde{R}_2}(|\widetilde{u}|^2 
    +|\widetilde{h}\nabla \widetilde{u}|^2)\,d\sigma
    =\frac{1}{2}\int_{B(0,\widetilde{R}-\eta,\widetilde{R}+\eta)}(|\widetilde{u}|^2 
    + |\widetilde{h}\nabla \widetilde{u}|^2)\,d x.
  \end{equation}
  Next, set $\widetilde{R}_1=\widetilde{R}/4$ and let
  $1_{B(0,\widetilde{R}_1-\eta/2,\widetilde{R}_1+\eta/2)}
  \prec\chi\in\mathcal{C}^{\infty}_c(B(0,\widetilde{R}_1-\eta,\widetilde{R}_1+\eta);[0,1])$.
  Then, there exist constants $c,\widetilde{c}>0$ such that
  \begin{equation*}
    \begin{split}
      \int_{B(0,\widetilde{R}_1-\eta,\widetilde{R}_1+\eta)}(|\widetilde{u}|^2
      + |\widetilde{h}\nabla \widetilde{u}|^2)dx &\geq
      c \int_{B(0,\widetilde{R}_1)}(|\chi \widetilde{u}|^2 + |\widetilde{h}\nabla \chi \widetilde{u}|^2)dx \\
      &\geq \widetilde{c} \int_{r=\widetilde{R}_1}|\widetilde{u}|^2
      \,d\sigma,
    \end{split}
  \end{equation*}
  where in the last inequality we use that the trace map
  $\tau:H^1(B(0,\widetilde{R}_1))\to L^2(\partial B(0,\widetilde{R}_1))$
  is continuous.  Similarly, using \eqref{eq3.41}, we get that
  \begin{equation}\label{eq3.44}
    \int_{r=\widetilde{R}_1}(| \widetilde{u}|^2 + |\widetilde{h}\nabla \widetilde{u}|^2)\,d\sigma \leq \mO(1)
    \int_{B(0,\widetilde{R}_1-\eta,\widetilde{R}_1+\eta)}(|\widetilde{u}|^2 + |\widetilde{h}\nabla \widetilde{u}|^2)dx. 
  \end{equation}
  Recall that $\lambda \in [a,b]\Subset\R\backslash\{0\}$, see
  \eqref{eq3.2.2}, and assume for simplicity that $a>0$.  Hence
  $a^{-2}\geq \lambda^{-2} \geq b^{-2} >0 $. Then, applying
  \eqref{eq3.42} with $\widetilde{R}_2$ and $\widetilde{R}_1$ as in
  \eqref{eq3.43} and \eqref{eq3.44} yields that there exist constants
  $C,C',\delta,h_0>0$ such that for any $\lambda\in I$ and any $0<h\leq h_0$
  \begin{equation*}
    \begin{split}
      \begin{split}
	-\Ima \int_{r=\widetilde{R}_2} \widetilde{h}\partial_r
        \widetilde{u}\cdot\overline{\widetilde{u}}\,d\sigma &\geq
        C\lambda \min\{1,b^{-2}\}
        \int_{r=\widetilde{R}_2}(|\widetilde{u}|^2
	+ |\widetilde{h}\nabla \widetilde{u}|^2)\,d\sigma \\
	& -
        \max\{1,a^{-2}\}C'\e^{-\delta|\lambda|/\widetilde{h}}\int_{r=\widetilde{R}_1}(|\widetilde{u}|^2
	+  |\widetilde{h}\nabla \widetilde{u}|^2)\,d\sigma \\
	& \geq \frac{Ca}{2} \min\{1,b^{-2}\}
        \int_{B(0,\widetilde{R}-\eta,\widetilde{R}+\eta)}(|\widetilde{u}|^2
	+  |\widetilde{h}\nabla \widetilde{u}|^2)\,d x \\
	& - \max\{1,a^{-2}\}C'\e^{-\delta a
          /\widetilde{h}}\int_{B(0,\widetilde{R}_1-\eta,\widetilde{R}_1+\eta)}(|\widetilde{u}|^2
        + |\widetilde{h}\nabla \widetilde{u}|^2)dx.
      \end{split}
    \end{split}
  \end{equation*}
  Then scaling back yields
  \begin{equation}\label{eq3.45}
    \begin{split}
      -h^{-1/3}\Ima \int_{r=\widetilde{R}_2h^{-1/3}} h\partial_r
      u\cdot\overline{u}\,d\sigma &\geq C_1
      \int_{A(\widetilde{R},\eta,h)}(|u|^2
      + |h\nabla u|^2)\,d x \\
      & -C_2 \e^{-\delta a /
        h^{4/3}}\int_{A(\widetilde{R}_1,\eta,h)}(|u|^2 + |h\nabla
      u|^2)\,dx,
    \end{split}
  \end{equation}
  for some constants $C_1,C_2>0$. By \eqref{eq3.50} we get that
  \begin{equation*}
    \begin{split}
      \int_{B(0,\widetilde{R}_2h^{-1/3})} v\overline{u} dx &=
      \int_{B(0,\widetilde{R}_2h^{-1/3})} (P_V-\lambda^2)u \cdot \overline{u} dx \\
      &= \int_{B(0,\widetilde{R}_2h^{-1/3})} ( (V-\lambda^2)|u|^2 +
      |h\nabla u|^2 ) dx - \int_{r=\widetilde{R}_2h^{-1/3}}
      h^2\partial_r u\cdot\overline{u}\,d\sigma.
    \end{split}
  \end{equation*}
  Taking the imaginary part yields that
  \begin{equation*}
    - \Ima \int_{r=\widetilde{R}_2h^{-1/3}} h\partial_r u\cdot\overline{u}\,d\sigma 
    =  h^{-1}\Ima (v | u ).
  \end{equation*}
  This together with \eqref{eq3.45} yields the statement of Lemma
  \ref{lem:Out}.
\end{proof}
\section{Resonance free region}
In this section we give a proof of Theorem \ref{thm:2} and show that
away from $0$ there are no resonances super-exponentially close to the
real axis.
The proof is standard and can be found for instance in \cite{Sj02ln}.
We will present it here for the reader's sake. The principal idea is
that assuming \eqref{thm:1.1} we can extend the resolvent
$(P_{V}-\mu^2)^{-1}$ holomorphically to $\mu$ in an exponentially
small disc centered at $\lambda$ as an operator
$L^2_{comp} \to H^2_{loc}$.
\\
\par
Here, we are only interested in the poles of the resolvent close to
the real axis. Therefore, let $\lambda \in I\Subset \R\backslash\{0\}$ so
that \eqref{thm:1.1} is valid and let $\Omega\subset\C\backslash i \R$
be a complex open neighborhood of $I$ such that the resolvent
\begin{equation*}
  R_V(\mu)\defeq (P_{V}-\mu^2)^{-1}, \quad \mu \in \Omega
\end{equation*}	
is holomorphic for $\Ima \mu >0$ and continues meromorphically to
$\Omega$, see Theorem \ref{thm:Mc2}. Next, notice that for
$\mu\in \Omega$
\begin{equation}\label{B1}
  R_V(\mu)= R_0(\mu)(1+Q(\mu))^{-1},
\end{equation}	
where $Q(\mu) = VR_0(\mu)$. This expression makes sense since it holds
for $\Ima \mu >0$ and by analytic Fredholm theory $(1+Q(\mu))^{-1}$
continues meromorphically from $\Ima \mu >0$ to $\mu\in\Omega$. To see
this let first $\Ima\mu>0$. Since $\Omega$ does not contain any
discrete spectrum of $P_V$ - which is situated on $i\R_+$ in the $\mu$
variable - we have that $Q(\mu)$ is a holomorphic family compact
operator $L^2\to L^2$ for $\Ima\mu>0$. Recall \eqref{eq3.2.1}, let
$R>R_0$ and let $1_{B(0,R_0)}\prec\chi \prec 1_{B(0,R)}$. Recall from
Theorem \ref{thm:Mc2} that
$\chi R_0(\mu)\chi: L^2(\R^d) \to H^2_0(B(0,R))$ is a holomorphic family
of operators for $\mu\in\Omega$. Hence, by the Rellich-Kondrachov
theorem $\chi R_0(\mu)\chi$ is a holomorphic family of compact
operators $L^2 \to L^2$.  Since $V=V\chi$ it follows that $1+ Q\chi$ is a
holomorphic family of Fredholm operators $L^2 \to L^2$ for
$\mu\in\Omega$. Since $(1+ Q\chi)^{-1}$ for $\Ima\mu \gg 1$ exists by a
Neumann series argument, it follows by analytic Fredholm theory that
$(1+ Q\chi)^{-1}:L^2\to L^2$ extends to a meromorphic family of Fredholm
operators to $\mu\in \Omega$.
\par
Next, notice that $(1+Q)=(1+Q(1-\chi))(1+Q\chi)$ and that
$(1+Q(1-\chi))^{-1} =(1-Q(1-\chi))$ has a holomorphic extension from
$\Ima \mu >0$ to $\mu\in\Omega$ as an operator
$L^2_{\comp}(\R^d)\to L^2_{\comp}(\R^d)$. Hence, $(1+Q(\mu))^{-1}$ has a
meromorphic extension from $\Ima \mu >0$ to $\mu\in\Omega$ as an operator
$L^2_{\comp}(\R^d)\to L^2_{\comp}(\R^d)$ and, thus, \eqref{B1} holds.
\\
\par
%
%
Let us now turn to the proof of Theorem \ref{thm:2}. Suppose that 
$\Omega\Subset \C\backslash i\R$ is a relatively compact open 
complex neighborhood of the interval $I$. Let $\mu\in\Omega$, 
suppose that 
$\Ima \mu \geq 0$ and assume that $R> R_0 +1$ and let
$1_{B(0,R-1)}\prec \chi_0\prec \chi_1\prec\chi_2 \prec \chi_3\prec
1_{B(0,R)}$
with $\chi_j\in \mathcal{C}^{\infty}_c(\R^d)$. We approximate the
interior part of the resolvent $R_V(\mu)\chi_1$ by
\begin{equation*}
  A(\mu) \defeq \chi_2 R_V(\lambda) \chi_1 - R_0(\mu)[P_0,\chi_2] R_V(\lambda)\chi_1.
\end{equation*}	
Then,
\begin{equation}\label{eq_rf1}
  \begin{split}
    (P_V - &\mu^2)A(\mu) \\
    &= \chi_1 + [P_0,\chi_2]R_V(\lambda)\chi_1 +
    \chi_2(\lambda^2 - \mu^2)R_V(\lambda)\chi_1 - (1+VR_0(\mu))[P_0,\chi_2]R_V(\lambda)\chi_1\\
    &= \chi_1 + \chi_2(\lambda^2 - \mu^2)R_V(\lambda)\chi_1 -VR_0(\mu)[P_0,\chi_2]R_V(\lambda)\chi_1.
  \end{split}
\end{equation}
Next, set $u=R_V(\lambda)\chi_1$. Then, 
since $1_{\supp V}\prec1_{B(0,R-1)}\prec \chi_0\prec \chi_1\prec\chi_2 \prec \chi_3\prec
1_{B(0,R)}$, we see that 
\begin{equation*}
\begin{split}
	(P_0-\lambda^2)(1-\chi_2)u & = (1-\chi_2)\chi_1 - [P_0,\chi_2] u \\
						   &=- [P_0,\chi_2] R_V(\lambda)\chi_1,
\end{split}
\end{equation*}
which implies that 
\begin{equation}\label{eq_rf1.1.0}
	(1-\chi_2)u =- R_0(\lambda)[P_0,\chi_2] R_V(\lambda)\chi_1.
\end{equation}
A priori the above two expressions make sense for $\Ima \lambda >0$, however, 
by analytic continuation, they hold as well for $\lambda\in I$. 
\par
Next, notice that the support of the term on the right hand side of \eqref{eq_rf1.1.0}  
is contained in $\supp (1-\chi_2)$ which has empty intersection with 
the support of the potential $V$. Hence, 
\begin{equation*}
	VR_0(\lambda)[P_0,\chi_2]R_V(\lambda) \chi_1=0.
\end{equation*}
Thus, by \eqref{eq_rf1}, we deduce that 
\begin{equation}\label{eq_rf1.1}
\begin{split}
    &(P_V - \mu^2)A(\mu)
    = \chi_1 + \chi_2(\lambda^2-\mu^2)R_V(\lambda)\chi_1 + T\\ 
   & \text{ with } T \defeq 
     -V\chi_3(R_0(\mu)-R_0(\lambda))\chi_3[P_0,\chi_2]R_V(\lambda)\chi_1.
\end{split}
\end{equation}
By \eqref{mc3.1}, we have that $\chi_3R_0(\mu)\chi_3$, as an operator
from $ L^2(\R^d) \to L^2(\R^d)$, is of norm $\mO(\e^{C/h})$ uniformly for $\mu \in \Omega$. 
Then, the Cauchy inequalities imply that 
\begin{equation*}
	\| \partial_{\mu}(\chi_3R_0(\mu)\chi_3)\|_{L^2 \to L^2} 
	= \mO(\e^{C/h})
\end{equation*}
uniformly for $\mu\in\widetilde{\Omega}$, where $\widetilde{\Omega}\Subset \Omega$ 
is a slightly smaller complex open neighborhood of $I$ strictly contained in $\Omega$. 
Thus, for any $\lambda\in I$ and any $\mu\in \widetilde{\Omega}$,
\begin{equation}\label{eq_rf1.2}
	\|\chi_3R_0(\mu)\chi_3-\chi_3R_0(\lambda)\chi_3\|_{L^2 \to L^2}  
	= \mO(|\mu-\lambda|\e^{C/h}).
\end{equation}

By \eqref{thm:1.1b}, we see that 
\begin{equation*}
	\|[P_0,\chi_2]R_V(\lambda)\chi_1\|_{L^2 \to L^2} = \mO\!\left(\e^{Ch^{-4/3}\log \frac{1}{h}}\right),
\end{equation*}
which in combination with \eqref{eq_rf1.1}, \eqref{eq_rf1.2} and \eqref{int1.0} gives that 
\begin{equation}\label{eq_rf1.3}
	\|T \|_{L^2 \to L^2}  =  \mO\!\left(|\mu-\lambda|\e^{Ch^{-4/3}\log \frac{1}{h}}\right).
\end{equation}
Notice that $\supp T\subset \supp V \subset B(0,R_0)$, which yields that $T$ maps 
$L^2(\R^d) \to L^2_{\comp}(B(0,R))$.
\\
\par
For the exterior part of the resolvent $R_V(\mu)(1-\chi_1)$ we use the
approximation
\begin{equation*}
  B(\mu) \defeq (1-\chi_0) R_0(\mu)(1- \chi_1) 
  + A(\mu)[P_0,\chi_0]R_0(\mu)(1-\chi_1). 
\end{equation*}	
Then,
\begin{equation}\label{eq_rf2}
  \begin{split}
    (P_V - \mu^2)B(\mu)
    &= 1-\chi_1 - [P_0,\chi_0]R_0(\mu)(1-\chi_1) \\
    &\phantom{=}
    +   (\chi_1 + \chi_2(\lambda^2 - \mu^2)R_V(\lambda)\chi_1+T)[P_0,\chi_0]R_0(\mu)(1-\chi_1)\\
    &= (1-\chi_1) + (\chi_2(\lambda^2 -
    \mu^2)R_V(\lambda)+T)[P_0,\chi_0]R_0(\mu)(1-\chi_1).
  \end{split}
\end{equation}
Here, we used as well that $\chi_0\prec \chi_1$. 
\\
\par
Put $\widetilde{R}(\mu)\defeq A(\mu) + B(\mu):~ L^2_{\comp}(\R^d)
\to H^2_{\loc}(\R^d)$. Then, combining \eqref{eq_rf1.1} and \eqref{eq_rf2} gives 
\begin{equation}\label{eq_rf3}
  (P_V - \mu^2)\widetilde{R}  = 1 + K 
\end{equation}
with
\begin{equation}\label{eq_rf4}
  K =  \chi_2(\lambda^2 - \mu^2)R_V(\lambda)(\chi_1 + [P_0,\chi_0]R_0(\mu)(1-\chi_1)) 
  	+T(1+[P_0,\chi_0]R_0(\mu)(1-\chi_1)). 
\end{equation}
%
%
Using \eqref{mc3.1} we have that $[P_0,\chi_0]R_0(\mu)$ as an operator
from $ L^2_{\comp}(B(0,R)) \to L^2_{\comp}(B(0,R))$ is of norm
$\mO(\e^{C/h})$ uniformly for $\mu \in \Omega$.
%
%
It then follows by \eqref{thm:1.1} and \eqref{eq_rf1.3} that
\begin{equation}\label{eq_rf5}
  K: ~ L^2_{\comp}(B(0,R)) \longrightarrow L^2_{\comp}(B(0,R))
\end{equation}
has operator norm
$ \leq \mO(|\lambda-\mu|\, \e^{Ch^{-4/3}\log h^{-1}})$ for some
constant $C>0$. Therefore, if
$|\lambda -\mu| \leq  \e^{-2Ch^{-4/3}\log h^{-1}}$, for $h>0$ small
enough, it follows that $(1+K)$ has a bounded inverse
\begin{equation}\label{eq_rf8.1}
  (1+K)^{-1} :~ L^2_{\comp}(B(0,R)) \longrightarrow L^2_{\comp}(B(0,R))
\end{equation}
and we get that
\begin{equation}\label{eq_rf8}
  R_V(\mu) = \widetilde{R}(\mu)(1+K)^{-1}: 
  ~ L^2_{\comp}(B(0,R)) \longrightarrow H^2_{\loc}(\R^d)
\end{equation}
is holomorphic for $|\lambda -\mu| \leq  \e^{-2Ch^{-4/3}\log h^{-1}}$ for
$h>0$ small enough. For $\mu$ still in the same set, it follows by
\eqref{B1} that
\begin{equation*}
\begin{split}
    R_V(\mu) &= R_0(\mu) - R_V(\mu)Q(\mu) \\
    & = R_0(\mu) - R_V(\mu)\chi_3Q(\mu).
  \end{split}
\end{equation*}
Since both $Q(\mu): L^2_{\comp}(\R^d)\to  L^2_{\comp}(\R^d)$ and
$R_0(\mu): L^2_{\comp}(\R^d)\to  H^2_{\loc}(\R^d)$ are holomorphic
families of operators, it follows by \eqref{eq_rf8} that
\begin{equation}\label{eq_rf9}
  R_V(\mu) : 
  ~ L^2_{\comp}(\R^d) \longrightarrow H^2_{\loc}(\R^d)
\end{equation}
is holomorphic for $|\lambda -\mu| \leq  \e^{-2Ch^{-4/3}\log h^{-1}}$, for
$h>0$ small enough, which completes the proof of Theorem \ref{thm:2}.
\section{Decay of eigenfunctions of Schr{\"o}dinger operators with bounded
  potentials}
In this section we prove Theorem \ref{thm3}. Let $d\geq 2 $, let
$0\not\equiv W\in L^{\infty}(\R^d)$ with $\|W\|_{\infty} \leq C_W$ and let
$u$ be a bounded solution to
\begin{equation}\label{d1}
  -\Delta u + W u = 0 \quad \text{in } \R^d
\end{equation}
and suppose that $u$ admits the estimate $|u(x)|\leq C\exp(-c|x|)$ for
$|x|>1$ and some constants $C,c>0$. Notice that in particular
$u\in H^2_h(\R^d)$, the semiclassical Sobolev space.  Let $h\in (0,1]$ and
let $U_{h}: L^2(\R^d) \to L^2(\R^d)$ be the unitary map defined by
\begin{equation}\label{d1.1}
  (U_{h}\phi)(x) = h^{d/2}\phi(hx). 
\end{equation}
Then,
\begin{equation}\label{d2}
  U_h^*(-\Delta + W) U_h = -h^2 \Delta + W(h^{-1}x) .
\end{equation}
Let $\psi \in \mathcal{C}^{\infty}_c(\R;[0,1])$ be such that
$\supp \psi \subset (1+1/4,1+1/2)$ and $\int \psi dx =1$. Then, set
\begin{equation*}
  \chi_h(x) \defeq \chi_h(|x|) \defeq 1 - \int_0^{|x|}\psi \left(1+\frac{t-1}{h}\right) dt.
\end{equation*}
Notice that $\chi_h \in\mathcal{C}^{\infty}_c(\R^d;[0,1])$ with
support contained in the ball $B(0,2)$ independently of
$h>0$. Moreover, $\chi_h\equiv 1$ on $\overline{B(0,1+h/4)}$ and
$\chi_h = 0$ outside $B(0,1+h/2)$. For any $\alpha\in\N^d\backslash \{0\}$
we have that the support of $\partial^{\alpha}\chi_h $ is contained in the
annulus $B(0,1+h/4,1+h/2)$ with inner radius $1+h/4$ and outer radius
$1+h/2$ and all derivatives satisfy the estimate
$\|\partial^{\alpha}\chi_h\|_{\infty} = \mO(h^{1-|\alpha|})$ for $\alpha\neq 0$.
Similarly, we can construct a
$\widetilde{\chi}_h \in\mathcal{C}^{\infty}_c(\R^d;[0,1])$ so that
$\widetilde{\chi}_h\equiv 1$ on $\supp \nabla \chi_h$ and
$\widetilde{\chi}_h = 0$ outside the annulus $B(0,1,1+h)$. Moreover,
we can arrange so that all derivatives satisfy the estimate
$\|\partial^{\alpha}\widetilde{\chi}_h\|_{\infty} = \mO(h^{1-|\alpha|})$, for
$\alpha\neq 0$.
\par
Set
\begin{equation*}
  \widetilde{u}(x) = (U_h^*u)(x).
\end{equation*}
Then, by \eqref{d1}, \eqref{d2}
\begin{equation}\label{d3}
  \begin{split}
    (-h^2 \Delta &+ W(h^{-1}x)\mathbf{1}_{B(0,2)}(x) )\chi_h(x) \widetilde{u}(x)  \\
    &= \chi_h(x)(-h^2 \Delta + W(h^{-1}x))\widetilde{u}(x)+ [-h^2\Delta,\chi_h(x)]\widetilde{u}(x)\\
    &=[-h^2\Delta,\chi_h(x)] \widetilde{u}(x).
  \end{split}
\end{equation}
%
Notice that
\begin{equation*}
  \begin{split}
    |[-h^2\Delta,\chi_h] \widetilde{u}|^2
    &= | (-h^2\Delta\chi_h)\widetilde{u} -2 h\nabla\chi_h \cdot h\nabla \widetilde{u}|^2\\
    &\leq 2| (-h^2\Delta\chi_h)\widetilde{u}|^2 + 8 | h\nabla\chi_h|^2 |
    h\nabla \widetilde{u}|^2.
  \end{split}
\end{equation*}
Using that that $\widetilde{\chi}_h\equiv 1$ on
$\supp \nabla \chi_h\subset B(0,1+h/4,1+h/2)$ and the estimate on its
derivatives, we see by integration by parts shows that
\begin{equation}\label{d4}
  \begin{split}
    \int |[-h^2\Delta,\chi_h] \widetilde{u}|^2 dx &\leq \mO(h^2)
    \int_{B(0,1,1+h)} | \widetilde{u}|^2 dx
    + \mO(h^2) \int_{\supp\nabla \chi_h} |h\nabla  \widetilde{u}|^2  \,dx \\
    &\leq \mO(h^2) \int_{B(0,1,1+h)} | \widetilde{u}|^2 dx
    + \mO(h^2) \int_{B(0,1,1+h)} |h^2\Delta  \widetilde{u}|^2  \,dx \\
    &\leq \mO(h^2) \int_{B(0,1,1+h)} | \widetilde{u}|^2 dx,
  \end{split}
\end{equation}
where in the last line we used as well that
$-h^2\Delta \widetilde{u} = - W(h^{-1}x)\widetilde{u}$ by \eqref{d1},
\eqref{d2}. Hence, setting
$V(x;h) = W(h^{-1}x)\mathbf{1}_{B(0,2)}(x) \in L^{\infty}(\R^d)$ with
$\supp V \subset B(0,2)$, we get by \eqref{d3}, \eqref{d4} and \eqref{d1.1}
that
\begin{equation}\label{d5}
  \| (-h^2 \Delta +V)\chi_h\widetilde{u}\|^2  =
  \mO(h^2) \int_{B(0,h^{-1},h^{-1}+1)} | u|^2 dx \defeq \varepsilon(h)
\end{equation}
Next, we apply Lemma \ref{lem:3.1} with $R=3$: there exists a
real-valued smooth function $\phi\in\mathcal{C}^{\infty}(\R^d)$ and
constants $C>0$ and $h_0\in (0,1]$ such that for any
$v\in\mathcal{C}^{\infty}_c(B(0,3))$ and all $0<h \leq h_0$
\begin{equation}\label{d6}
  \int \e^{2\phi/h^{4/3}}(|v|^2 + |h\nabla v|^2 ) dx 
  \leq \frac{C}{h^{2/3}}
  \int \e^{2\phi/h^{4/3}}| (-h^2 \Delta +V)v|^2 dx. 
\end{equation}	
Notice in particular from the proof of Lemma \ref{lem:3.1} that
$\phi = \max\{\|V\|_{\infty},1\}^{2/3} \phi_0 $ where $\phi_0$ is a
smooth real-valued function which does not depend on the potential $V$
as it stems from the Carleman estimate for the free Laplacian. In fact
$\phi_0$ is a non-constant function since one requires
$|d\phi_0| \neq 0$ for the Carleman estimate to work, see for instance
\cite{Sj02ln}.  Furthermore, since we assume that $W\not\equiv 0$ we
obtain by an easy modification of the proof of Lemma \ref{lem:3.1}
that we can take $\phi = \|V\|_{\infty}^{2/3} \phi_0 $ for $h>0$ small
enough. Let $M\defeq \max_{B(0,3)}\phi_0 - \min_{B(0,3)}\phi_0 >0$,
then, applying \eqref{d6} to $\chi_h\widetilde{u}$, we get in
combination with \eqref{d5} that
\begin{equation}
  \int_{B(0,h^{-1})}|u|^2 dx \leq Ch^{4/3} \e^{ 2M \|V\|_{\infty}^{2/3}/h^{4/3}}
  \int_{B(0,h^{-1},h^{-1}+1)} | u|^2 dx .
\end{equation}	
Since we assumed that $u$ that $\|u\|_2 =1$, we get that 
for $h>0$ small enough 
\begin{equation}
  \int_{B(0,h^{-1},h^{-1}+1)} | u|^2 dx  \geq 2C h^{-4/3} \e^{ - 2M \|V\|_{\infty}^{2/3}/h^{4/3}}.
\end{equation}
Setting $R=h^{-1}$, we conclude formula \eqref{RL1} and hence the
proof of Theorem \ref{thm3}.
%
%
%
%
%
\def\cprime{$'$} 

\providecommand{\bysame}{\leavevmode\hbox to3em{\hrulefill}\thinspace}
\providecommand{\MR}{\relax\ifhmode\unskip\space\fi MR }
\providecommand{\MRhref}[2]{%
  \href{http://www.ams.org/mathscinet-getitem?mr=#1}{#2}
}
\providecommand{\href}[2]{#2}


\begin{thebibliography}{NTTV18}

\bibitem[BTV17]{BTV17}
D.~Borisov, M.~Tautenhahn, and I.~Veseli\'c, \emph{Scale-free quantitative
  unique continuation and equidistribution estimates for solutions of elliptic
  differential equations}, J. Math. Phys. \textbf{58} (2017), no.~12.

\bibitem[Bur98]{Bu98}
N.~Burq, \emph{{D{\'e}croissance de l'{\'e}nergie locale de l'{\'e}quation des
  ondes pour le probl{\`e}me ext{\'e}rieur et absence de r{\'e}sonance au
  voisinage du r{\'e}el}}, Acta Math{\'e}matica \textbf{180} (1998), 1--29.

\bibitem[Bur02]{Bu02}
\bysame, \emph{{Lower bounds for shape resonances widths of long range
  Schr{\"o}dinger operators}}, Amer. J. Math. (2002), no.~4, 677--735.

\bibitem[CHK07]{CHK07}
J.-M. Combes, P.~D. Hislop, and F.~Klopp, \emph{An optimal w egner estimate and
  its application to the global continuity of the integrated dens ity of states
  for random schr{\"o}dinger operators}, Duke Math. J. \textbf{140} (2007),
  no.~3, 469--498.

\bibitem[CV02]{CaVo02}
F.~Cardoso and G.~Vodev, \emph{Uniform estimates of the resolvent of the
  laplace-beltrami operator on infinite volume riemannian manifolds ii}, Ann.
  Henri Poincar{\'e} \textbf{3} (2002), no.~4, 673--691.

\bibitem[Dat14]{Da14}
K.~Datchev, \emph{Quantitative limiting absorption principal in the
  semiclassical limit}, Geom. Funct. Anal. \textbf{24} (2014), 740--747.

\bibitem[DKW17]{DKW17}
B.~Davey, C.~Kenig, and J.-N. Wang, \emph{The landis conjecture for variable
  coefficient second-order elliptic pdes}, Transaction of the AMS (2017),
  no.~11, 8209--8237.

\bibitem[DZ]{DZ17+}
S.~Dyatlov and M.~Zworski, \emph{Mathematical theory of scattering resonances},
  preprint.

\bibitem[JB05]{BoKe05}
C.~E.~Kenig J.~Bourgain, \emph{On localization in the continuous
  anderson-bernoulli model in higher dimension}, Invent. math. \textbf{161}
  (2005), 389--426.

\bibitem[KL88]{KoLa88} V.~A. Kondrat\cprime ev and E.~M. Landis,
  \emph{Qualitative theory of linear partial differential equations of
    second order}, Itogi Nauki i Tekhniki: Sovremennye Problemy Mat.:
  Fundamental'nye Napravleniya \textbf{32} (1988), 99--218.

\bibitem[Kle23]{Kl13}
A.~Klein, \emph{Unique continuation principle for spectral projections of
  schr{\"o}dinger operators and optimal wegner estimates for non-ergodic random
  schr{\"o}dinger operators}, Comm. Math. Phys \textbf{232} (20123), no.~3,
  1229--1246.

\bibitem[Klo16]{Kl16}
F.~Klopp, \emph{Resonances for large one-dimensional "ergodic" systems}, Anal.
  PDE \textbf{9} (2016), no.~2, 259--352.

\bibitem[Mes92]{Me92}
V.~Z. Meshkov, \emph{On the possible rate of decay at infinity of solutions of
  second order partial differential equations}, Math. USSR Sbornik \textbf{72}
  (1992), no.~343, 343--351.

\bibitem[NTTV15]{NTTV15}
N.~Naki\'c, M.~T{\"a}ufer, M.~Tautenhahn, and I.~Veseli\'c, \emph{Scale-free
  uncertainty principles and wegner estimates for random breather potentials.},
  C. R. Math. Acad. Sci. Paris \textbf{353} (2015), no.~10, 919--923.

\bibitem[NTTV18]{NTTV18}
\bysame, \emph{Scale-free unique continuation principle for spectral
  projectors, eigenvalue-lifting and wegner estimates for random
  schr{\"o}dinger operators}, Anal. PDE \textbf{11} (2018), no.~4, 1049--1081.

\bibitem[RMV13]{RMV13}
C.~Rojas-Molina and I.~Veseli\'c, \emph{Scale-free unique continuation
  estimates and applications to random schr{\"o}dinger operators}, Comm. Math.
  Phys. \textbf{320} (2013), no.~1, 245--274.

\bibitem[Sha16]{Sh16}
J.~Shapiro, \emph{Semiclassical resolvent bounds in dimension two}, to appear
  in Proc. of the AMS, pre-print \url{https://arxiv.org/abs/1604.03852} (2016).

\bibitem[Sha18]{Sh18+}
\bysame, \emph{Semiclassical resolvent bound for compactly supported
  $l^{\infty}$ potentials}, pre-print \url{https://arxiv.org/abs/1802.09008}
  (2018).

\bibitem[Sj{\"o}02]{Sj02ln}
J.~Sj{\"o}strand, \emph{{Lectures on resonances}}, 2002.

\bibitem[TV15]{TV15}
M.~Tautenhahn and I.~Veseli\'c, \emph{Discrete alloy-type models: regularity of
  distributions and recent results}, Markov Process. Related Fields \textbf{21}
  (2015), no.~3, 823--846.

\bibitem[Vod00]{Vo00}
G.~Vodev, \emph{Exponential bounds of the resolvent for a class of noncompactly
  supported perturbations of the laplacian.}, Mathematical Research Letters,
  \textbf{7} (2000), 287--298.

\bibitem[Vod14]{Vod14}
\bysame, \emph{Semi-classical resolvent estimates and regions free of
  resonances}, Math. Nachr. \textbf{287} (2014), no.~7, 825--835.

\end{thebibliography}
\end{document}